\documentclass[letterpaper, reqno, 11pt]{aq-amsart}

\title{Weighted linearization of vector fields\\via a formal Moser trick}
\author{Arthur Lei Qiu}

\newcommand*{\euler}{\mathcal{E}}
\newcommand*{\lin}{\mathrm{lin}}
\newcommand*{\tdvf}{\vf_t}
\newcommand*{\vfFilt}[1]{\vf_{[\![#1]\!]}}
\newcommand*{\vfGrad}[1]{\vf_{[#1]}}
\newcommand*{\poly}{\mathrm{pol}}

\DeclareMathOperator{\id}{id}                          
\DeclareMathOperator{\eval}{eval}                      

\newcommand*{\N}{\mathbb{N}}                           
\newcommand*{\R}{\mathbb{R}}                           
\newcommand*{\C}{\mathbb{C}}                           
\newcommand*{\K}{\mathbb{K}}
\let\hungarumlaut\H                                    
\renewcommand*{\H}{
    \TextOrMath{\hungarumlaut}{\mathbb{H}}%
}

\newcommand*{\p}{\partial}                             
\def\XXint#1#2#3{
    {\setbox0=\hbox{$#1{#2#3}{\int}$ }
    \vcenter{\hbox{$#2#3$ }}\kern-.6\wd0}
}

\DeclareMathOperator{\Mat}{Mat}
\newcommand*{\transpose}{t}                            
\newcommand{\pair}[1]{\left\langle{#1}\right\rangle}   

\newcommand*{\coordvf}[2][]{
    \frac{\partial}{\partial{#2}}%
    \ifstrempty{#1}{}{\mathchoice                      
    {\Big\rvert_{#1}}
    {\rvert_{#1}}
    {\rvert_{#1}}
    {\scriptstyle\rvert_{#1}}
    }
}
\newcommand*{\vf}{\mathfrak{X}}                        
\newcommand*{\lie}{\mathcal{L}}                        
\newcommand*{\tensor}[1][]{%
    \if\relax\detokenize{#1}\relax
        \mathcal{T}%
    \else
        \mathcal{T}^{(#1)}%
    \fi
}

\DeclareMathOperator{\ad}{ad}                          

\let\slant\sl                                          
\renewcommand*{\sl}{%
    \TextOrMath{\slant}{\mathfrak{sl}}%
}

\let\slashO\O                                          
\renewcommand*{\O}{
    \TextOrMath{\slashO}{\mathrm{O}}%
}
\let\slasho\o                                          
\renewcommand*{\o}{%
    \TextOrMath{\slasho}{\mathfrak{o}}%
}

\let\ubreve\u                                          
\renewcommand*{\u}{%
    \TextOrMath{\ubreve}{\mathfrak{u}}%
}

\newcommand*{\powseries}[2]{#1[\![#2]\!]}              
\newcommand*{\maxideal}{\mathfrak{m}}                  
\DeclareMathOperator{\Der}{Der}                        
\DeclareMathOperator{\Hom}{Hom}                        

\newcommand*{\f}{\frac}                                

\begin{document}

\maketitle

\begin{abstract}
    Many well-known theorems establish sufficient criteria for linearizability of a vector field in terms of the eigenvalues of its linear approximation. By attaching weights to coordinates so that some directions are considered ``linear'', others ``quadratic'', and so on, one can define the notion of a \emph{weighted linear approximation}. It is thus natural to ask when a vector field is ``weighted-linearizable''. In this paper, we formulate a weighted version of the non-resonance condition appearing in the Poincar\'{e} and Sternberg linearization theorems and show that it implies weighted linearizability.

    Our approach first addresses weighted linearization on the level of formal power series. In doing so, we develop a general framework to make sense of a power series version of \emph{Moser's trick}, a technique used to prove various normal form results in geometry. This formal Moser trick works over any field of characteristic zero and may be of independent interest.
\end{abstract}

\setcounter{tocdepth}{1}
\tableofcontents

\section{Introduction}\label{sec:introduction}

\subsection{Background}\label{subsec:background}

Linearization of vector fields is a fundamental technique with applications throughout mathematics, science, and engineering. It features prominently in the study of local normal forms in differential geometry and dynamical systems, where a natural question arises: when does the first-order approximation furnished by linearization accurately reflect a vector field's geometry and dynamics near a fixed point? This question is classical, dating back to Poincar\'{e} and Lie, and by now there is a wealth of literature establishing sufficient conditions for linearizability with respect to various notions of ``accuracy''.

In the topological category, the Hartman--Grobman theorem \cite[\S13]{arnold_geometrical_1988} states that the flow of a $C^1$ vector field on $\R^n$ can be conjugated via a homeomorphism to the flow of its linear approximation near a \emph{hyperbolic} fixed point (one where the vector field's Jacobian matrix has no purely imaginary eigenvalues). On the level of formal power series, Poincar\'{e}'s thesis \cite[Part II]{poincare_sur_1879} investigates when a system of $n$ ordinary differential equations
\begin{align*}
    \dot{x} = Ax + O(|x|^2), \qquad A \in \Mat_n(\C)
\end{align*}
can be brought to the form $\dot{y} = Ay$ by a formal change of coordinates $y = x + O(|x|^2)$. He formulates the following sufficient condition: call the eigenvalues $\lambda_1, \ldots, \lambda_n$ of the matrix $A$ \emph{resonant} if there exists $j \in \{1, \ldots, n\}$ and non-negative integers $\alpha_1, \ldots, \alpha_n$ such that
\begin{align*}
    \pair{\lambda, \alpha} = \lambda_j \quad\text{and}\quad \sum_{i = 1}^n \alpha_i > 1, \quad\text{where}\quad\pair{\lambda, \alpha} := \sum_{i = 1}^n \lambda_i\alpha_i.
\end{align*}
Poincar\'{e}'s linearization theorem states that, if the eigenvalues are \emph{not} resonant, then such a formal change of coordinates exists. Should the formal power series $Ax + O(|x|^2)$ define a holomorphic vector field on $\C^n$, results of Poincar\'{e} and Siegel \cite[\S24]{arnold_geometrical_1988} describe when a linearizing formal coordinate transformation converges to a biholomorphism, again in terms of eigenvalues of $A$.

In the smooth (by which we always mean $C^\infty$) category, Sternberg's linearization theorem \cite{sternberg_local_1957, sternberg_structure_1958} states that a smooth vector field on $\R^n$ can be linearized near a fixed point $p$ by a diffeomorphism provided that the (possibly complex) eigenvalues of its Jacobian matrix at $p$ satisfy the same non-resonance condition as in Poincar\'{e}'s linearization theorem. For example, fix an integer $k > 1$ and consider the vector field on $\R^2$ given by
\begin{align*}
    X = (x + y^k)\coordvf{x} + y\coordvf{y}.
\end{align*}
This vector field is \textit{Euler-like}, meaning its linear approximation at $p = 0$ is the \emph{Euler vector field}
\begin{align*}
    \euler = \tilde{x}\coordvf{\tilde{x}} + \tilde{y}\coordvf{\tilde{y}}.
\end{align*}
Since $\euler$ has non-resonant eigenvalues $(\lambda_1, \lambda_2) = (1, 1)$, Sternberg's theorem ensures that there exists a smooth local change of coordinates $(x, y) \mapsto (\tilde{x}, \tilde{y})$ transforming $X$ into $\euler$ near $0$. In fact, one can explicitly check that the (global) change of coordinates $\tilde{x} = x - \f{1}{k - 1}y^k$, $\tilde{y} = y$ works. More generally, a vector field on a smooth manifold is called Euler-like for a submanifold $P$ if it is tangent to $P$ and its linear approximation along $P$ equals the Euler vector field of the normal bundle of $P$; Bursztyn--Lima--Meinrenken have proved that every Euler-like vector field is smoothly linearizable along the entire submanifold \cite[Lemma 2.4]{bursztyn_splitting_2019}.

In contrast, neither Sternberg's nor Poincar\'{e}'s theorem can be applied to the \emph{weighted Euler-like} vector field
\begin{align}\label{eq:weighted-euler-like-example}
    Y = (x + y^k)\coordvf{x} + 2y\coordvf{y},
\end{align}
as the eigenvalues $(\lambda_1, \lambda_2) = (1, 2)$ of its linear part satisfy the resonance relation $2\lambda_1 = \lambda_2$. Nevertheless, $Y$ can be smoothly linearized: the coordinates $x' = x - \f{1}{2k - 1}y^k$, $y' = y$ take $Y$ to the \emph{weighted Euler vector field}
\begin{align*}
    \euler_w = x'\coordvf{x'} + 2y'\coordvf{y'}.
\end{align*}
Actually, this coordinate transformation does the job even when $k = 1$, in which case the eigenvalues of the linear part remain resonant but the linear term $y\coordvf{x}$ disappears upon changing coordinates. More generally, Meinrenken has proved that every weighted Euler-like vector field is smoothly linearizable \cite[Lemma 2.5]{meinrenken_euler-like_2021}, even though many such vector fields are resonant.

Thus, while the non-resonance condition is sufficient for formal and smooth linearizability, it is far from necessary. For resonant vector fields, Poincar\'{e}'s theorem is extended by the Poincar\'{e}--Dulac theorem \cite[\S23]{arnold_geometrical_1988}, which produces a formal change of coordinates that might remove some but not all non-linear terms in the vector field. The particular non-linearities that can appear upon changing coordinates depend on certain non-canonical choices made during the iterative construction of the coordinate transformation, and these in turn depend on the particular resonance relations occurring in the vector field, so there is some ambiguity in what the resulting normal form looks like.

This paper offers one systematic way to gain more control: by accepting a generalized notion of linearity based on assigning positive integer weights to the coordinates of $\R^n$ or $\C^n$, so that coordinates of weight $1$ are considered ``linear'', coordinates of weight $2$ are considered ``quadratic'', and so on. The weighted Euler vector field $\euler_w$ corresponds to assigning weights $w = (1, 2)$ to the first and second coordinates of $\R^2$, respectively, which explains why the coordinate transformation $(x, y) \mapsto (x', y')$ removes the term $y^k\coordvf{x}$ from the weighted Euler-like vector field $Y$ even when $k = 1$: it is considered ``non-linear'' with respect to these weights.

Such \emph{weightings} appear throughout mathematics. For example, weighted projective spaces are objects of study in algebraic geometry, and Loizides--Meinrenken \cite{loizides_differential_2023} have recently developed a differential-geometric theory of weightings in detail. Weighted linear approximations of vector fields have been previously studied by Algaba--Garc\'{i}a--Reyes \cite{algaba_like-linearizations_2009} under the name of ``like-linearizations''. Weighted Euler-like vector fields play an important role in their work, which provides necessary and sufficient conditions for a vector field to be equivalent to its weighted linear approximation in two different senses: the vector fields themselves being conjugate, and their flows being orbitally equivalent (i.e., the same up to time reparametrization).

\subsection{Main results}\label{subsec:main-results}

Algaba--Garc\'{i}a--Reyes' characterization of weighted linearizability differs from the other linearization theorems discussed in Section \ref{subsec:background} in that it is not expressed in terms of eigenvalues, which are straightforward to compute in principle. Our goal is to give an eigenvalue-based sufficient condition akin to the non-resonance condition of Poincar\'{e}'s and Sternberg's linearization theorems. This \emph{weighted non-resonance} condition is more flexible than the Poincar\'{e}--Sternberg condition and implies weighted linearizability; the compromise is that weighted linear approximations need not be linear in the usual sense. Nevertheless, one has more control over the non-linear terms that can appear compared to the Poincar\'{e}--Dulac theorem, as any remaining non-linearities must be ``weighted-linear''. Moreover, one can choose the weighting to their advantage, and sometimes by doing so one obtains weighted linear approximations that \emph{are} linear in the usual sense, as is the case for the vector field \eqref{eq:weighted-euler-like-example} with the weighting $w = (1, 2)$.

By combining the main results of this paper, we offer the following weighted generalization of the Poincar\'{e} and Sternberg linearization theorems.

\begin{theorem}\label{thm:main-theorem}
    Let $X$ be either a formal vector field on $\K^n$ where $\K$ is any field of characteristic zero, or a smooth vector field on an open neighbourhood of $0 \in \R^n$. If $w$ is any weighting of $\K^n$ (resp. $\R^n$) with respect to which $X$ is ``admissible'' and ``non-resonant'' (Definitions \ref{def:weighted-lin} and \ref{def:weighted-resonance}), then there exists a formal (resp. germ of a smooth) diffeomorphism that conjugates $X$ with its weighted linear approximation with respect to $w$.
\end{theorem}

Taking the \emph{trivial weighting} $w = (1, \ldots, 1)$ recovers Poincar\'{e}'s and Sternberg's theorems, for in this case admissibility simply means that $0$ is a fixed point, non-resonance coincides with the usual notion, and weighted linear approximation is just linear approximation. A special case of Theorem \ref{thm:main-theorem} (Proposition \ref{prop:weighted-euler-like-lin}) easily detects that the $y^k\coordvf{x}$ term in the vector field \eqref{eq:weighted-euler-like-example} can be removed when Poincar\'{e}'s and Sternberg's theorems cannot (though one already knows this from \cite[Lemma 2.5]{meinrenken_euler-like_2021}).

To explain our approach to proving Theorem \ref{thm:main-theorem}, it is worth stepping back from weighted linearizability to consider an abstracted problem. A fundamental question in differential geometry is: given tensor fields (or differential forms, or multivector fields) $A_0$ and $A_1$ on a smooth manifold $M$, when does there exist a diffeomorphism $\varphi \colon M \to M$ such that the pullback $\varphi^*A_1$ is equal to $A_0$? If not globally, then one can ask: supposing $x \in M$ is a point such that $A_0(x) = A_1(x)$, when does there exist a locally defined diffeomorphism $\varphi$ such that $\varphi(x) = x$ and $\varphi^*A_1 = A_0$ nearby?

\emph{Moser's trick} is a powerful technique for answering this question. In the broadest sense, it consists of searching for a one-parameter family of tensor fields (or differential forms, or multivector fields) $\{A_t\}_{t \in [0, 1]}$ joining $A_0$ and $A_1$, and a one-parameter family of diffeomorphisms $\{\varphi_t\}_{t \in [0, 1]}$ such that $\varphi_0$ is the identity map and $\f{d}{dt}(\varphi_t^*A_t) = 0$. If such families can be found, then $\varphi_1^*A_1 = \varphi_0^*A_0 = A_0$, so $\varphi := \varphi_1$ solves the problem. Typically, instead of directly looking for the family of diffeomorphisms $\varphi_t$, one seeks out the time-dependent vector field $X_t$ that generates $\varphi_t$ as its flow. The rate of change of $\varphi_t^*A_t$ is related to the Lie derivative $\lie_{X_t}$ via the identity
\begin{align}\label{eq:pullback-rate-of-change}
    \f{d}{dt}(\varphi_t^*A_t) = \varphi_t^*\left(\f{dA_t}{dt} + \lie_{X_t} A_t\right).
\end{align}

One can often make a judicious choice of the interpolating family $A_t$ based on the particular structure of a given problem, thus easily finding a solution $X_t$ to the differential equation
\begin{align}\label{eq:lie-deriv-diff-eq}
    \f{dA_t}{dt} + \lie_{X_t} A_t = 0.
\end{align}
Integrating $X_t$ to its flow then produces the desired $\varphi_t$. Originally developed by Moser in the context of volume-preserving geometry \cite{moser_volume_1965}, Moser's trick can be used to give very short proofs of the Darboux theorem \cite{weinstein_symplectic_1969} and Morse lemma \cite{palais_morse_1969}. It also facilitates Bursztyn--Lima--Meinrenken's proof that Euler-like vector fields are linearizable along submanifolds, and Meinrenken's proof that weighted Euler-like vector fields are conjugate to their corresponding weighted Euler vector fields.

Techniques developed by Sternberg were later modified by Chen \cite{chen_equivalence_1963} to prove a ``theorem of equivalence'': two smooth vector fields, both with a hyperbolic fixed point at $0 \in \R^n$, are smoothly conjugate near $0$ if and only if they are formally conjugate. Since complex roots of polynomials with real coefficients come in conjugate pairs, non-resonance implies hyperbolicity for smooth vector fields. (If a fixed point is \emph{not} hyperbolic, then either $0$ is an eigenvalue of the Jacobian matrix, in which case $0 = 2 \cdot 0$ is a resonance, or there is a pair of non-zero purely imaginary eigenvalues $\pm i\theta$, in which case $i\theta = 2(i\theta) + 1(-i\theta)$ is a resonance.) Thus, one can linearize a non-resonant smooth vector field by first Taylor expanding its component functions; Poincar\'{e}'s theorem then produces a formal diffeomorphism that kills off any higher finite-order terms; Chen's theorem upgrades this to an actual diffeomorphism that kills off any remaining infinite-order terms.

In view of Chen's theorem of equivalence, we sought to develop a Moser-type argument for Theorem \ref{thm:main-theorem} that passes through the world of formal power series. Working with functions and vector fields on $\R^n$ or $\C^n$ only insofar as their formal Taylor series mattered led us to na\"{i}vely wonder if the time parameter $t$ in Moser's trick could be treated as a formal variable as well. One obvious problem arises: namely, how to ``evaluate'' an arbitrary power series at $t = 1$ (or any $t \neq 0$ for that matter). Rather than giving up and reverting to treating $t$ as an actual parameter, we obstinately proceeded to develop criteria describing when it \emph{does} make sense to ``evaluate at non-zero $t$'' in a way that ended up not using any properties about $\R$ or $\C$ besides being fields of characteristic zero. We thus obtain a ``Moser trick for formal power series'' in contexts beyond its original scope:

\begin{named*}[Formal Moser trick]
    Let $\K$ be any field of characteristic zero, and let $A_0$ and $A_1$ be formal tensor fields on $\K^n$. Suppose there exists an ``evaluative'' (Definition \ref{def:evaluative}) time-dependent formal tensor field $A_t$ interpolating between $A_0$ and $A_1$, and there exists a time-dependent formal vector field $X_t$ satisfying Equation \eqref{eq:lie-deriv-diff-eq} in a certain formal sense and whose flow is ``evaluative''. Then there exists a formal diffeomorphism $\varphi$ of $\K^n$ such that $\varphi^*A_1 = A_0$.
\end{named*}

\subsection{Paper overview}

The rest of this paper is structured as follows. In Sections \ref{subsec:formal-objects} and \ref{subsec:time-dependence}, we define algebraic analogues of standard differential-geometric concepts: vector/tensor fields, diffeomorphisms, related constructions, and one-parameter families thereof. In Section \ref{subsec:formal-flows}, we state and prove a version of the existence and uniqueness theorem for formal solutions to ordinary differential equations (Theorem \ref{thm:flow-existence-uniqueness}), allowing us to import Equations \eqref{eq:pullback-rate-of-change} and \eqref{eq:lie-deriv-diff-eq} into this setting. In Section \ref{subsec:fixed-time-evaluation}, we introduce ``evaluative'' formal objects, state a marginally different-looking version of the formal Moser trick (Corollary \ref{cor:formal-moser}), and give a sufficient condition for flows of formal vector fields to be evaluative.

Section \ref{subsec:weights} focuses on weightings and the associated notions of quasi-homogeneity of functions and vector fields. Though exposited within an algebraic setting, we also pause to make comparisons to the differentiable world. Section \ref{subsec:weights-eval} generalizes the sufficient condition for ``evaluability'' of formal flows from Section \ref{subsec:fixed-time-evaluation} to account for such weightings.

In Section \ref{subsec:weighted-lin-approx}, we define the class of vector fields for which it is sensible to pose the question of weighted linearizability; sketch an informal argument yielding a plausibly sufficient condition in the differentiable world; then reverse the argument to prove that this condition is indeed enough for weighted linearizability on the level of formal power series (Theorem \ref{thm:weighted-linearizability}). Phrased in terms of invertibility of the adjoint operator defined by a given vector field's weighted linear approximation, this condition at first looks dissimilar to the eigenvalue-based conditions of the linearization theorems previously described. We bridge the gap in Section \ref{subsec:weighted-hom-eq} by stating our weighted non-resonance condition and proving it implies that this adjoint operator is invertible (Theorem \ref{thm:adjoint-invertibility}), thus yielding a weighted generalization of Poincar\'{e}'s theorem. The paper concludes in Section \ref{subsec:weighted-lin-smooth} by explaining how Chen's theorem of equivalence can be invoked in the weighted setting to upgrade from formal to smooth weighted linearizability (Corollary \ref{cor:smooth-weighted-lin}), thus generalizing Sternberg's theorem.

Throughout this paper, $\K$ denotes a fixed choice of any field of characteristic zero. All algebras and algebra homomorphisms are required to be unital; all derivations are required to be $\K$-linear. We include $0$ in the set of natural numbers $\N := \{0, 1, 2, \ldots\}$ to avoid overusing sub-/superscripts. A multi-index of length $k$ is a $k$-tuple $\alpha = (\alpha_1, \ldots, \alpha_k) \in \N^k$; its degree is denoted by $|\alpha| := \sum_i \alpha_i$.

\subsection{Acknowledgements}

I am greatly indebted to Eckhard Meinrenken for his diligent mentorship while supervising my Master's project, of which the present work is an expansion. I am also deeply grateful to Yael Karshon for insightful comments and patient support throughout the preparation of this paper. Feedback from Nic Fellini and Charlie Wu has been instrumental in the process of writing and editing the manuscript. Posthumous thanks are owed to Will J. Merry for the profound impact of his teachings on the way I think and write about mathematics.

This work was partially funded by a University of Toronto Excellence Award, an Ontario Graduate Scholarship, and a doctoral NSERC Canada Graduate Research Scholarship.

\section{Formal geometry}\label{sec:formal-geometry}

We begin by laying out the framework in which we shall obtain a formal version of Moser's trick. Much of the material in this section may be well-known in the language of $\infty$-jets, but we include it for the sake of making this paper self-contained and setting the notation used throughout. We have made an effort to highlight connections to various related concepts in geometry and algebra; in particular, to poly-flow vector fields and locally nilpotent derivations (Remarks \ref{rmk:poly-flow} and \ref{rmk:locally-nilpotent}).

\subsection{Formal functions, diffeomorphisms, and tensor fields}\label{subsec:formal-objects}

We write $\powseries{\K}{x} := \powseries{\K}{x^1, \ldots, x^n}$ for the $\K$-algebra of formal power series in $n$ commuting variables $x^1, \ldots, x^n$ with coefficients in $\K$. Elements of $\powseries{\K}{x}$ are formal infinite sums
\begin{align*}
    \sum_{\alpha \in \N^n} c_\alpha x^\alpha, \quad \text{where } c_\alpha \in \K \text{ and }x^\alpha := (x^1)^{\alpha_1} \cdots (x^n)^{\alpha_n}.
\end{align*}
Such expressions are intended to be evocative of the Taylor series of (germs of) smooth functions,\footnote{This is true on the nose for $\K = \R$, as the $\R$-algebra homomorphism $C^\infty(\R^n) \to \powseries{\R}{x}$ sending a smooth function to its formal Taylor series centred at the origin is surjective by Borel's lemma \cite[\S1.5]{narasimhan_analysis_1985}.} so we will often call them \textbf{formal functions} on $\K^n$ and denote them by $f$ or $f(x)$ interchangably. The ``constant term'' $c_{(0, \ldots, 0)}$ is then denoted by $f(0)$; the map $f \mapsto f(0)$ defines a surjective $\K$-algebra homomorphism $\powseries{\K}{x} \to \K$. We extend this notation to $n$-tuples of formal power series $\varphi(x) = \big(\varphi^1(x), \ldots, \varphi^n(x)\big)$ component-wise by setting $\varphi(0) := \big(\varphi^1(0), \ldots, \varphi^n(0)\big)$, and similarly for matrices with entries in $\powseries{\K}{x}$.

Since $\K$ is a local integral domain, so too is $\powseries{\K}{x}$. Its unique maximal ideal $\maxideal$ consists of those power series $f$ such that $f(0) = 0$; equivalently, $\maxideal$ is generated by the variables $x^1, \ldots, x^n$. We equip $\powseries{\K}{x}$ with the \emph{$\maxideal$-adic topology}, which is the translation-invariant topology such that the decreasing filtration by ideals
\begin{align*}
    \powseries{\K}{x} = \maxideal^0 \supset \maxideal^1 \supset \maxideal^2 \supset \cdots
\end{align*}
forms a basis of open neighbourhoods of the zero power series. We adopt the Big Theta notation
\begin{align*}
    \Theta(f) := \max\{k \mid f \in \maxideal^k\} = \min\{|\alpha| \mid c_\alpha \neq 0\}
\end{align*}
to denote the \emph{order} of a power series $f(x) = \sum_\alpha c_\alpha x^\alpha$, with the convention that the zero power series has order $\infty$. Since $\K$ is an integral domain, one has $\Theta(fg) = \Theta(f) + \Theta(g)$ \cite[IV\S4.8]{bourbaki_algebra_2003}. We extend this notation to $n$-tuples of power series $\varphi(x) = \big(\varphi^1(x), \ldots, \varphi^n(x)\big)$ by defining $\Theta(\varphi) := \min_i\Theta(\varphi^i)$.

The $\maxideal$-adic topology on $\powseries{\K}{x}$ is induced by a complete metric; for example, the one given by $(f, g) \mapsto 2^{-\Theta(f - g)}$. A sequence of power series $f_k(x) = \sum_\alpha c_{k, \alpha} x^\alpha$ is Cauchy with respect to some (equivalently, every) compatible complete metric if and only if, for each $\alpha$, the sequence of coefficients $c_{k, \alpha}$ is eventually constant; the coefficient of $x^\alpha$ in the limiting power series is then given by $\lim_{k \to \infty} c_{k, \alpha}$. An infinite sum $\sum_k^\infty f_k$ converges in the $\maxideal$-adic topology if and only if $f_k \to 0$, or equivalently, $\Theta(f_k) \to \infty$.

The $\K$-algebra $\powseries{\K}{x}$ together with its $\maxideal$-adic topology satisfies a certain universal property, of which we state a special case \cite[IV\S4.3]{bourbaki_algebra_2003}:
\begin{center}
    \emph{There is a bijective correspondence
    {\arraycolsep=0.0pt
    \begin{align}\label{eq:K[[x]]-univ-prop}
        \left\{\begin{array}{c}
            \text{continuous }\K\text{-algebra endomorphisms} \\
            \varphi \colon \powseries{\K}{x} \to \powseries{\K}{x}
        \end{array}\right\}
        \leftrightarrow
        \left\{
        \begin{array}{c}
            n\text{-tuples of elements of }\maxideal \\
            \left(\varphi^1(x), \ldots, \varphi^n(x)\right)
        \end{array}
        \right\},
    \end{align}
    }one determining the other by defining $\varphi(x^i) = \varphi^i(x)$ for each $i$.}
\end{center}
The power series $\varphi^i(x)$ is called the \emph{$i$\textsuperscript{th} component} of $\varphi$ (or the corresponding $n$-tuple). Continuity with respect to the $\maxideal$-adic topology ensures that the image of any power series under $\varphi$ may be computed ``term-by-term'', even with infinitely many non-zero terms. For example, one always has
\begin{align}\label{eq:homogeneous-polynomial-decomposition}
    \sum_{\alpha \in \N^n} c_\alpha x^\alpha = \sum_{k = 0}^\infty \Bigg(\sum_{|\alpha| = k} c_\alpha x^\alpha\Bigg),
\end{align}
where the left-hand side is an element of $\powseries{\K}{x}$ and the right-hand side is a convergent limit of partial sums of a sequence in $\powseries{\K}{x}$; continuity allows us to compute
\begin{align*}
    \varphi\bigg(\sum_\alpha c_\alpha x^\alpha\bigg) = \sum_{k = 0}^\infty \varphi\Bigg(\sum_{|\alpha| = k} c_\alpha x^\alpha\Bigg) = \sum_\alpha c_\alpha \varphi(x^1)^{\alpha_1}\cdots\varphi(x^n)^{\alpha_n}.
\end{align*}
In this sense, $\varphi$ can be thought of as ``pullback of formal functions'' (but see Remark \ref{rmk:pullback-notation}).

\begin{remark}\label{rmk:adic-continuity}
    Actually, for any power series algebras $\powseries{\K}{x} = \powseries{\K}{x^1, \ldots, x^n}$ and $\powseries{\K}{y} = \powseries{\K}{y^1, \ldots, y^m}$ over $\K$ (possibly in different sets of variables), every $\K$-algebra homomorphism $\varphi \colon \powseries{\K}{x} \to \powseries{\K}{y}$ is \emph{automatically} continuous with respect to the $\maxideal$-adic topologies, where $\maxideal$ refers to the maximal ideals in both $\powseries{\K}{x}$ and $\powseries{\K}{y}$. To see this, recall that a power series with coefficients in any commutative ring $R$ is invertible if and only if its constant term is invertible in $R$. For each $i \in \{1, \ldots, n\}$, let $c^i \in \K$ be the constant term of $\varphi(x^i)$. Since $\varphi$ is $\K$-linear and unital, we have
    \begin{align*}
        \varphi(x^i - c^i) = \varphi(x^i) - c^i \in \maxideal_y.
    \end{align*}
    Thus $\varphi(x^i - c^i)$ is not invertible in $\powseries{\K}{y}$; as $\varphi$ is a unital ring homomorphism, this implies that $x^i - c^i$ is not invertible in $\powseries{\K}{x}$. Since $\K$ is a field, we obtain $c^i = 0$. Thus $\varphi(x^i) \in \maxideal_y$ for each $i$; as the $x^i$'s generate $\maxideal_x$, we conclude that $\varphi(\maxideal_x) \subseteq \maxideal_y$, so $\varphi$ is $\maxideal$-adically continuous.

    This argument relies on $\K$ being a field. Over more general commutative rings, $\maxideal$-adically discontinuous homomorphisms of power series algebras can exist and are more complicated to describe; see \cite{eakin_r-endomorphisms_1976} and the references within. We explicitly mention continuity in the bijective correspondence \eqref{eq:K[[x]]-univ-prop} only to emphasize its utility in future computations.
\end{remark}

A diffeomorphism from a smooth manifold to itself defines an automorphism of its algebra of smooth functions via pullback. In this spirit, we define a \textbf{formal diffeomorphism}\footnote{Strictly speaking, what we call a ``formal diffeomorphism'' more precisely reflects the Taylor series of a smooth diffeomorphism at a \emph{fixed point}.} of $\K^n$ to be a $\K$-algebra automorphism of $\powseries{\K}{x}$. Thus, a formal diffeomorphism determines a unique $n$-tuple of elements of $\maxideal$ via the bijective correspondence \eqref{eq:K[[x]]-univ-prop}. Conversely, given any $n$-tuple $\big(\varphi^1(x), \ldots, \varphi^n(x)\big)$ of elements of $\maxideal$, there is a simple criterion to determine when the corresponding endomorphism $\varphi \colon \powseries{\K}{x} \to \powseries{\K}{x}$ is invertible (hence an automorphism). It mimics the inverse function theorem: $\varphi$ is an automorphism if and only if the Jacobian matrix of formal partial derivatives
\begin{align*}
    D\varphi(x) :=
    \begin{pmatrix}
        \frac{\p \varphi^1}{\p x^1}(x) & \ldots & \frac{\p \varphi^1}{\p x^n}(x) \\
        \vdots & \ddots & \vdots \\
        \frac{\p \varphi^n}{\p x^1}(x) & \ldots & \frac{\p \varphi^n}{\p x^n}(x)
    \end{pmatrix}
\end{align*}
is invertible in $\Mat_n(\powseries{\K}{x})$. As the constant term of $\det D\varphi(x)$ is $\det D\varphi(0)$, this is also equivalent to $D\varphi(0)$ being invertible in $\Mat_n(\K)$ \cite[IV\S4.7]{bourbaki_algebra_2003}.

Next, we develop the formal versions of smooth tensor fields. The goal is to mimic smooth expressions at every step; as such, the reader would not be led too far astray by importing the smooth version of each definition/construction and replacing the algebra of smooth functions with the algebra of formal functions everywhere. Nevertheless, we shall describe the main players to eliminate the guesswork needed and mention their salient features.

Smooth vector fields act as derivations of a manifold's algebra of smooth functions. Analogously, we define a \textbf{formal vector field} on $\K^n$ to be a derivation of the algebra of formal functions, i.e., a $\K$-linear map $X \colon \powseries{\K}{x} \to \powseries{\K}{x}$ satisfying the Leibniz rule:
\begin{align*}
    X(fg) = X(f)g + fX(g), \qquad \forall f, g \in \powseries{\K}{x}.
\end{align*}
We denote the set of formal vector fields by $\vf$; it is both a module over $\powseries{\K}{x}$ and a Lie algebra over $\K$ via the commutator Lie bracket $[X, Y] := X \circ Y - Y \circ X$. Every formal vector field is continuous with respect to the $\maxideal$-adic topology; a corollary of this fact is that $\vf$ is free of rank $n$ over $\powseries{\K}{x}$, with a basis given by the set of formal partial derivatives $\left\{\coordvf{x^1}, \ldots, \coordvf{x^n}\right\}$ \cite[IV\S4.6]{bourbaki_algebra_2003}.

We define a \textbf{formal covector field} on $\K^n$ to be an element of the dual $\powseries{\K}{x}$-module
\begin{align*}
    \Omega^1 := \Hom_{\powseries{\K}{x}}(\vf, \powseries{\K}{x}).
\end{align*}
Every formal function $f \in \powseries{\K}{x}$ has a \emph{formal differential} $df \in \Omega^1$, defined by $df(X) = X(f)$ for $X \in \vf$. The set $\{dx^1, \ldots, dx^n\}$ is the dual basis to $\left\{\coordvf{x^1}, \ldots, \coordvf{x^n}\right\}$. Thus, every $X \in \vf$ can be uniquely written in the form $X = \sum_i X^i\coordvf{x^i}$ where $X^i = X(x^i) = dx^i(X)$, and every $\omega \in \Omega^1$ can be uniquely written in the form $\omega = \sum_i \omega_i\,dx^i$ where $\omega_i = \omega(\coordvf{x^i})$. In particular, $df = \sum_i \f{\p f}{\p x^i}\,dx^i$. (See Remark \ref{rmk:kahler-differential-comparison} for a comparison of formal covector fields and K\"{a}hler differentials.)

Having defined formal vector and covector fields, we define a \textbf{formal tensor field of type $(r, s)$} on $\K^n$ to be an element of the tensor product of $\powseries{\K}{x}$-modules
\begin{align*}
    \tensor[r, s] := \underbrace{\vf \otimes \cdots \otimes \vf}_{r \text{ times}} \otimes \underbrace{\Omega^1 \otimes \cdots \otimes \Omega^1}_{s \text{ times}},
\end{align*}
with the convention that $\tensor[0, 0] := \powseries{\K}{x}$. Since $\vf$ and $\Omega^1$ are free of finite rank, formal tensor fields $A \in \tensor[r, s]$ are canonically identifiable with $\powseries{\K}{x}$-multilinear maps
\begin{align*}
    A \colon \underbrace{\Omega^1 \times \cdots \times \Omega^1}_{r \text{ times}} \times \underbrace{\vf \times \cdots \times \vf}_{s \text{ times}} \to \powseries{\K}{x}.
\end{align*}
The tensor product operation $\otimes$ endows
\begin{align*}
    \tensor := \bigoplus_{r, s \geq 0} \tensor[r, s]
\end{align*}
with the structure of a graded $\powseries{\K}{x}$-algebra. A \textbf{formal tensor derivation} is a $\K$-linear map $\mathcal{D} \colon \tensor \to \tensor$ that preserves type, commutes with all contraction homomorphisms $\tensor[r, s] \to \tensor[r - 1, s - 1]$ (defined by feeding one of the $\vf$ factors to one of the $\Omega^1$ factors), and satisfies
\begin{align*}
    \mathcal{D}(A \otimes B) = \mathcal{D}(A) \otimes B + A \otimes \mathcal{D}(B), \qquad \forall A, B \in \tensor.
\end{align*}
Two formal tensor derivations are identical if they agree on formal functions and formal vector fields. Conversely, every pair of $\K$-linear maps $\mathcal{D}^{(0, 0)} \colon \powseries{\K}{x} \to \powseries{\K}{x}$ and $\mathcal{D}^{(1, 0)} \colon \vf \to \vf$ satisfying
\begin{align*}
    \mathcal{D}^{(0, 0)}(fg) & = \mathcal{D}^{(0, 0)}(f)g + f\mathcal{D}^{(0, 0)}(g), & \mathcal{D}^{(1, 0)}(fX) & = \mathcal{D}^{(0, 0)}(f)X + f\mathcal{D}^{(1, 0)}(X)
\end{align*}
for all $f, g \in \powseries{\K}{x}$ and $X \in \vf$ extends uniquely to a formal tensor derivation. In particular, given $X \in \vf$, the \textbf{formal Lie derivative} $\lie_X \colon \tensor \to \tensor$ is defined to be the unique formal tensor derivation extending $\lie_X^{(0, 0)} f := X(f)$ and $\lie_X^{(1, 0)} Y := [X, Y]$. The map $X \mapsto \lie_X$ is $\K$-linear.

Just as diffeomorphisms pull back smooth tensor fields, a formal diffeomorphism $\varphi \colon \powseries{\K}{x} \to \powseries{\K}{x}$ extends to a graded $\K$-algebra automorphism $\varphi^* \colon \tensor \to \tensor$, called the \textbf{pullback}, by requiring $\varphi^*$ to distribute across $\otimes$ and commute with $d \colon \powseries{\K}{x} \to \Omega^1$ and all contractions. That is, for formal functions $f \in \tensor[0, 0]$, define $\varphi^*f = \varphi(f)$. For formal vector fields $X \in \tensor[1, 0]$, define $\varphi^*X \in \tensor[1, 0]$ to be the composition
\begin{align*}
    \powseries{\K}{x} \xrightarrow{\varphi^{-1}} \powseries{\K}{x} \xrightarrow{X} \powseries{\K}{x} \xrightarrow{\varphi} \powseries{\K}{x},
\end{align*}
so that $(\varphi^*X)(\varphi^*f) = \varphi^*\big(X(f)\big)$ for all $f \in \powseries{\K}{x}$. For formal covector fields $\omega \in \tensor[0, 1]$, define $\varphi^*\omega \in \tensor[0, 1]$ to be the composition
\begin{align*}
    \vf \xrightarrow{(\varphi^*)^{-1}} \vf \xrightarrow{\omega} \powseries{\K}{x} \xrightarrow{\varphi} \powseries{\K}{x},
\end{align*}
so that $(\varphi^*\omega)(\varphi^*X) = \varphi^*\big(\omega(X)\big)$ for all $X \in \vf$. For formal tensor fields of arbitrary type $A \in \tensor[r, s]$, define $\varphi^*A \in \tensor[r, s]$ as a $\powseries{\K}{x}$-multilinear map by requiring that
\begin{align*}
    (\varphi^*A)(\varphi^*\omega_1, \ldots, \varphi^*\omega_r, \varphi^*X_1, \ldots, \varphi^*X_s) = \varphi^*\big(A(\omega_1, \ldots, \omega_r, X_1, \ldots, X_s)\big)
\end{align*}
for all $\omega_1, \ldots, \omega_r \in \Omega^1$ and $X_1, \ldots, X_s \in \vf$.

A few comparisons to the smooth setting assure us that these formal analogues are sensibly defined. Firstly, $\varphi^*$ restricts to a Lie algebra automorphism $\vf \to \vf$. Secondly, pullbacks are compatible with Lie derivatives: for any $X \in \vf$, the composition $\varphi^* \circ \lie_X \circ (\varphi^*)^{-1}$ is a formal tensor derivation that agrees with $\lie_{\varphi^*X}$ on formal functions and formal vector fields, so they must agree on all formal tensor fields; that is,
\begin{align}\label{eq:pullback-lie-derivative}
    \varphi^*(\lie_X A) = \lie_{\varphi^*X}(\varphi^*A), \qquad \forall A \in \tensor.
\end{align}
Lastly, pullbacks look as expected ``in coordinates'', which we demonstrate with formal vector fields $X = \sum_i X^i\coordvf{x^i}$ and formal covector fields $\omega = \sum_i \omega_i\,dx^i$. If we denote by $(\varphi^*X)(x) \in \powseries{\K}{x}^n$ the $n$-tuple of components of $\varphi^*X$ with respect to the basis $\{\coordvf{x^i}\}$ and we denote by $X\big(\varphi(x)\big)$ the $n$-tuple $\big(\varphi(X^1), \ldots, \varphi(X^n)\big) \in \powseries{\K}{x}^n$, then since $\varphi(X^i) = (\varphi^*X)(\varphi^*x^i) = \sum_j (\varphi^*X)(x^j)\f{\p \varphi^i}{\p x^j}$, one finds that
\begin{align*}
    X\big(\varphi(x)\big) & = D\varphi(x)\big[(\varphi^*X)(x)\big].
\end{align*}
As $D\varphi(x)$ is invertible in $\Mat_n(\powseries{\K}{x})$, this can be rewritten as
\begin{align*}
    (\varphi^*X)(x) = D\varphi(x)^{-1}\big[X\big(\varphi(x)\big)\big],
\end{align*}
mimicking the usual pointwise formula for pullbacks of smooth vector fields along diffeomorphisms. Similarly, if we denote by $(\varphi^*\omega)(x) \in \powseries{\K}{x}^n$ the $n$-tuple of components of $\varphi^*\omega$ with respect to the basis $\{dx^i\}$ and we denote by $\omega\big(\varphi(x)\big)$ the $n$-tuple $\big(\varphi(\omega_1), \ldots, \varphi(\omega_n)\big) \in \powseries{\K}{x}^n$, then by computing $(\varphi^*\omega)\left(\coordvf{x^i}\right)$, one finds that $\varphi^*\omega$ is given by transposition:
\begin{align*}
    (\varphi^*\omega)(x) = D\varphi(x)^\transpose\big[\omega\big(\varphi(x)\big)\big].
\end{align*}

\begin{remark}[terminology and notation]\label{rmk:pullback-notation}
    The ``coordinate'' calculations above justify why the map on formal tensor fields $\varphi^* \colon \tensor \to \tensor$ is called ``pullback'' and denoted with an upper star. But if $\psi$ is another formal diffeomorphism, then so is $\psi \circ \varphi$, and the pullback maps satisfy $(\psi \circ \varphi)^* = \psi^* \circ \varphi^*$---in apparent conflict with the usual convention that upper stars denote contravariant actions. We will have little need to consider compositions of formal diffeomorphisms, so this should not pose a serious issue. Nevertheless, this can be reconciled with the usual convention by recognizing that formal diffeomorphisms are defined by how they act on formal functions, not on ``formal points''. That is, our notation for pullback is contravariant \emph{with respect to maps of the spectrum of $\powseries{\K}{x}$.}
\end{remark}

\begin{remark}[formal covector fields versus K\"{a}hler differentials]\label{rmk:kahler-differential-comparison}
    Our definitions of $\Omega^1$ and the formal differential $d \colon \powseries{\K}{x} \to \Omega^1$ differ from the $\powseries{\K}{x}$-module of K\"{a}hler differentials $\Omega_{\powseries{\K}{x}/\K}$ and the universal derivation $\delta \colon \powseries{\K}{x} \to \Omega_{\powseries{\K}{x}/\K}$ in a few ways. Firstly, $\Omega_{\powseries{\K}{x}/\K}$ is neither projective nor finitely generated as a $\powseries{\K}{x}$-module \cite[Example 5.5a]{kunz_kahler_1986} whereas our $\Omega^1$ is free of finite rank. Secondly, $d$ satisfies $df = \sum_i \f{\p f}{\p x^i}\,dx^i$ for all $f \in \powseries{\K}{x}$ whereas---intuitively speaking---$\delta$ is only finitely additive, so one should not expect that $\delta f = \sum_i \f{\p f}{\p x^i}\,\delta x^i$ when $f$ has infinitely many non-zero coefficients. Thirdly, the pairs $(\Omega^1, d)$ and $(\Omega_{\powseries{\K}{x}/\K}, \delta)$ satisfy different universal properties \cite[Example 12.7]{kunz_kahler_1986}.

    To explain the relationship, recall the universal property of $(\Omega_{\powseries{\K}{x}/\K}, \delta)$: for every $\powseries{\K}{x}$-module $M$, pre-composition with $\delta$ defines an isomorphism of $\powseries{\K}{x}$-modules
    \begin{align*}
        \delta^\vee \colon \Hom_{\powseries{\K}{x}}(\Omega_{\powseries{\K}{x}/\K}, M) \to \Der_\K(\powseries{\K}{x}, M), \qquad T \mapsto T \circ \delta.
    \end{align*}
    Taking $M = \powseries{\K}{x}$, we see that $\delta^\vee$ identifies the dual $\powseries{\K}{x}$-module $(\Omega_{\powseries{\K}{x}/\K})^\vee$ with $\vf$, and thus $\delta^{\vee\vee}$ identifies $\Omega^1$ with the double dual $(\Omega_{\powseries{\K}{x}/\K})^{\vee\vee}$. As $d \in \Der_\K(\powseries{\K}{x}, \Omega^1)$, there is a unique $T \in \Hom_{\powseries{\K}{x}}(\Omega_{\powseries{\K}{x}/\K}, \Omega^1)$ such that $d = T \circ \delta$. If $\varepsilon \colon \Omega_{\powseries{\K}{x}/\K} \to (\Omega_{\powseries{\K}{x}/\K})^{\vee\vee}$ denotes the canonical homomorphism, given by $\varepsilon(\omega)(\chi) = \chi(\omega)$ for $\omega \in \Omega_{\powseries{\K}{x}/\K}$ and $\chi \in (\Omega_{\powseries{\K}{x}/\K})^\vee$, then by unravelling definitions one obtains the following commutative diagram:
    \begin{equation*}
        \begin{tikzcd}
            \powseries{\K}{x} \arrow[r, "d"] \arrow[d, swap, "\delta"] & \Omega^1 \arrow[d, "\delta^{\vee\vee}"] \\
            \Omega_{\powseries{\K}{x}/\K} \arrow[r, swap, "\varepsilon"] \arrow[ur, "T"] & (\Omega_{\powseries{\K}{x}/\K})^{\vee\vee}
        \end{tikzcd}
    \end{equation*}
    Thus, $T$ is just the canonical homomorphism $\varepsilon$ up to the identification provided by $\delta^{\vee\vee}$.
\end{remark}

\subsection{Formal time dependence}\label{subsec:time-dependence}

We now generalize the formal objects defined so far by allowing them to depend on another formal variable $t$ in addition to the $x^i$'s; the development will closely mirror Section \ref{subsec:formal-objects}. To begin, we introduce the notation $\powseries{\K}{t, x} := \powseries{\K}{t, x^1, \ldots, x^n}$ for the $\K$-algebra of ``time-dependent'' formal functions, from which all other ``time-dependent'' formal objects shall be built. Analogously to $\powseries{\K}{x}$, it is equipped with the $\maxideal_{t, x}$-adic topology where $\maxideal_{t, x}$ is the (unique maximal) ideal generated by $t$ and the $x^i$'s. Thus, any time-dependent formal function
\begin{align*}
    f(t, x) = \sum_{k, \alpha} c_{k, \alpha} t^k x^\alpha
\end{align*}
is in particular equal to the convergent limit of partial sums $\sum_k t^kf_k(x)$ where
\begin{align*}
    f_k(x) = \sum_\alpha c_{k, \alpha}x^\alpha \in \powseries{\K}{x} \text{ for each }k \in \N,
\end{align*}
and we can interchange such infinite sums with continuous operations on $\powseries{\K}{t, x}$.

We distinguish $t$ from the $x^i$'s notationally and psychologically because we will eventually try to ``evaluate time-dependent formal objects at a fixed time'' by replacing $t$ with an element of $\K$ and applying the field operations of $\K$ to make sense of the resulting expressions. The simplest example of this is the (continuous by Remark \ref{rmk:adic-continuity}) $\powseries{\K}{x}$-algebra homomorphism
\begin{align*}
    \eval_0 \colon \powseries{\K}{t, x} \to \powseries{\K}{x}, \qquad f(t, x) \mapsto f(0, x) = f_0(x).
\end{align*}
Of course, substituting non-zero values for $t$ is rarely well-defined, even if $\K$ is topologized to help make sense of the resulting infinite sums---for example, one cannot hope to define an ``evaluation at time $t = 1$'' map on all of $\powseries{\K}{t, x}$, as demonstrated by $\sum_k t^k$. However, we must initially allow arbitrary power series in $t$ to obtain an existence and uniqueness theorem for flows of formal vector fields (Theorem \ref{thm:flow-existence-uniqueness}); thus we develop the following story with all concepts based on $\powseries{\K}{t, x}$.

We define a \textbf{formal isotopy} of $\K^n$ to be a $\K$-algebra automorphism $\varphi \colon \powseries{\K}{t, x} \to \powseries{\K}{t, x}$ such that $\varphi(t) = t$. To see that this is a good formal analogue of a one-parameter family of diffeomorphisms, let us explicitly characterize formal isotopies similarly to how we characterized formal diffeomorphisms. The universal property of $\powseries{\K}{t, x}$ yields a bijective correspondence
{\arraycolsep=0.0pt
\begin{align*}
    \left\{\begin{array}{c}
        \text{(automatically continuous) }\K\text{-algebra} \\
        \text{endomorphisms } \varphi \colon \powseries{\K}{t, x} \to \powseries{\K}{t, x}
    \end{array}\right\}
    \leftrightarrow
    \left\{
    \begin{array}{c}
        (n + 1)\text{-tuples }\left(\varphi^0(t, x), \varphi^1(t, x), \ldots, \varphi^n(t, x)\right) \\
        \text{such that } \varphi^i(0, 0) = 0 \text{ for each } i \in \{0, 1, \ldots, n\}
    \end{array}
    \right\},
\end{align*}
}one determining the other via $\varphi(t) = \varphi^0$ and $\varphi(x^i) = \varphi^i$ for $i \in \{1, \ldots, n\}$. This restricts to a bijective correspondence between $\K$-algebra automorphisms of $\powseries{\K}{t, x}$ and such $(n + 1)$-tuples satisfying the additional condition that the $(n + 1) \times (n + 1)$ Jacobian matrix of partial derivatives of $(\varphi^0, \varphi^1, \ldots, \varphi^n)$ with respect to $t$ and the $x^i$'s is invertible. If $\varphi^0(t, x) = t$ (i.e., the corresponding endomorphism fixes $t$), then the top row of this matrix is $\begin{pmatrix} 1 & 0 & \cdots & 0 \end{pmatrix}$, so this matrix is invertible if and only if the bottom-right $n \times n$ submatrix
\begin{align*}
    D_x\varphi(t, x) :=
    \begin{pmatrix}
        \f{\p \varphi^1}{\p x^1}(t, x) & \cdots & \frac{\p \varphi^1}{\p x^n}(t, x) \\
        \vdots & \ddots & \vdots \\
        \f{\p \varphi^n}{\p x^1}(t, x) & \ldots & \frac{\p \varphi^n}{\p x^n}(t, x)
    \end{pmatrix}
\end{align*}
is invertible in $\Mat_n(\powseries{\K}{t, x})$, or equivalently, $D_x\varphi(0, 0)$ is invertible in $\Mat_n(\K)$. Knowing that $\varphi^0(t, x) = t$, we might as well omit it from such $(n + 1)$-tuples; this yields a bijective correspondence
{\arraycolsep=0.0pt
\begin{align}\label{eq:isotopy-characterization}
    \left\{\begin{array}{c}
        \text{formal isotopies} \\
        \varphi \colon \powseries{\K}{t, x} \to \powseries{\K}{t, x}
    \end{array}\right\}
    \leftrightarrow
    \left\{
    \begin{array}{c}
        n\text{-tuples } \varphi(t, x) \in \powseries{\K}{t, x}^n \text{ such that} \\
        \varphi(0, 0) = 0 \text{ and } D_x\varphi(0, 0) \text{ is invertible}
    \end{array}
    \right\}.
\end{align}
}Compare with how formal diffeomorphisms correspond to certain $n$-tuples in $\powseries{\K}{x}^n$: here, the $n$ component power series are allowed to depend on $t$ but the relevant Jacobian matrix still involves partial derivatives only with respect to the $x^i$'s. We will typically decorate $\K$-algebra endomorphisms of $\powseries{\K}{t, x}$ with a subscript $t$ (as in $\varphi_t$) to emphasize the allowed dependence on $t$.

\begin{remark}\label{rmk:isotopy-tuple-equivalence}
    Any $n$-tuple $\varphi(t, x) \in \powseries{\K}{t, x}^n$ can be written in the form $\varphi(t, x) = \sum_{k = 0}^\infty t^k\varphi_k(x)$ where $\varphi_k(x) \in \powseries{\K}{x}^n$; written this way, $\varphi(t, x)$ corresponds to a formal isotopy if and only if $\Theta(\varphi_0) \geq 1$ and $D\varphi_0(0)$ is invertible in $\Mat_n(\K)$.
\end{remark}

Every $\K$-algebra endomorphism of $\powseries{\K}{x}$ (resp. formal diffeomorphism) extends to a ``time-independent'' $\K$-algebra endomorphism of $\powseries{\K}{t, x}$ (resp. formal isotopy) by declaring it to fix $t$. On the level of $n$-tuples, this is simply the inclusion $\powseries{\K}{x}^n \subset \powseries{\K}{t, x}^n$. Conversely, we define the \emph{$t = 0$ time slice} of any $\K$-algebra endomorphism $\varphi_t$ of $\powseries{\K}{t, x}$ that fixes $t$ by the composition
\begin{align*}
    \varphi_0 \colon \powseries{\K}{x} \hookrightarrow \powseries{\K}{t, x} \xrightarrow{\varphi_t} \powseries{\K}{t, x} \xrightarrow{\eval_0} \powseries{\K}{x}.
\end{align*}
The subscript $0$ notation is consistent with identifying $\varphi_t$ with an $n$-tuple written as in Remark \ref{rmk:isotopy-tuple-equivalence}, as $\varphi_0$ corresponds to the $n$-tuple $\varphi_0(x) \in \powseries{\K}{x}^n$. Time zero slices are compatible with composition:

\begin{lemma}\label{lem:time-zero-slice-composition}
    For all $\K$-algebra endomorphisms $\varphi_t, \psi_t$ of $\powseries{\K}{t, x}$ fixing $t$, we have $(\psi_t \circ \varphi_t)_0 = \psi_0 \circ \varphi_0$. In particular, if $\varphi_t$ is a formal isotopy, then $\varphi_0$ is a formal diffeomorphism and $(\varphi_0)^{-1} = (\varphi_t^{-1})_0$.
\end{lemma}

\begin{proof}
    Denote the inclusion by $\imath \colon \powseries{\K}{x} \hookrightarrow \powseries{\K}{t, x}$. As $\eval_0 \circ\,\imath$ is the identity map on $\powseries{\K}{x}$, we have
    \begin{align*}
        \eval_0 \circ\,\varphi_t \circ \imath = \varphi_0 = \eval_0 \circ\,\imath \circ \varphi_0.
    \end{align*}
    Thus, the image of $\varphi_t \circ \imath - \imath \circ \varphi_0$ is contained in the kernel of $\eval_0$. Since $\psi_t$ is continuous and fixes $t$, it maps $\ker(\eval_0)$ to itself, so for all $f \in \powseries{\K}{x}$,
    \begin{align*}
        (\psi_t \circ \varphi_t)_0(f) - (\psi_0 \circ \varphi_0)(f) & = \eval_0\left[\psi_t\Big((\varphi_t \circ \imath)(f) - (\imath \circ \varphi_0)(f)\Big)\right] = 0.
    \end{align*}
    Thus $(\psi_t \circ \varphi_t)_0 = \psi_0 \circ \varphi_0$.
\end{proof}

\begin{remark}[terminology]
    Given an algebra $A$, algebra homomorphisms $A \to \powseries{A}{t}$ are sometimes called \emph{Hasse--Schmidt derivations} of $A$. Thus, any algebra endomorphism of $\powseries{\K}{t, x}$ can be regarded as a Hasse--Schmidt derivation of $\powseries{\K}{x}$ by restricting its domain, and any Hasse--Schmidt derivation of $\powseries{\K}{x}$ extends to an algebra endomorphism of $\powseries{\K}{t, x}$ by declaring it to fix $t$. We prefer to avoid this terminology, as the results to come are more naturally expressed in terms of self-maps of $\powseries{\K}{t, x}$ than in terms of maps $\powseries{\K}{x} \to \powseries{\K}{t, x}$.
\end{remark}

Before defining time-dependent formal tensor fields on $\K^n$, we note that by treating $t$ as ``just another formal variable'' no different from the $x^i$'s, one obtains formal tensor fields on $\K^{n + 1}$: the $\powseries{\K}{t, x}$-module of derivations of $\powseries{\K}{t, x}$, its dual, tensor products of these, and the associated definitions of Lie derivatives and pullbacks (along formal diffeomorphisms of $\K^{n + 1}$). Morally speaking, time-dependent formal tensor fields on $\K^n$ are simply formal tensor fields on $\K^n$, except that when written in terms of tensor products of $\coordvf{x^i}$'s and $dx^i$'s, their component power series are also allowed to depend on $t$. Equivalently, they are formal tensor fields on $\K^{n + 1}$ that do not involve $\coordvf{t}$ or $dt$.

A ``coordinate-free'' way to implement this is by declaring a \textbf{time-dependent formal tensor field of type $(r, s)$} on $\K^n$ to be an element of the $\powseries{\K}{t, x}$-module obtained from the $\powseries{\K}{x}$-module $\tensor[r, s]$ via extension of scalars along the inclusion $\imath \colon \powseries{\K}{x} \hookrightarrow \powseries{\K}{t, x}$:
\begin{align*}
    \tensor[r, s]_t := \powseries{\K}{t, x} \otimes_{\powseries{\K}{x}} \tensor[r, s].
\end{align*}
We denote by $\tdvf := \tensor[1, 0]_t$ and $\Omega^1_t := \tensor[0, 1]_t$ the time-dependent formal vector and covector fields, respectively, which are free $\powseries{\K}{t, x}$-modules with bases $\left\{\coordvf{x^1}, \ldots, \coordvf{x^n}\right\}$ and $\{dx^1, \ldots, dx^n\}$ by construction. In particular, $\tdvf$ is identifiable with the set of derivations of $\powseries{\K}{t, x}$ that vanish on $t$, which is a $\powseries{\K}{t, x}$-submodule and Lie subalgebra of the set of all derivations of $\powseries{\K}{t, x}$. Since $\tensor[r, s]$ is free of finite rank over $\powseries{\K}{x}$, the definition of $\tensor[r, s]_t$ is compatible with thinking of it as $(\tdvf)^{\otimes r} \otimes (\Omega^1_t)^{\otimes s}$ (the tensor products being taken over $\powseries{\K}{t, x}$), or as $\powseries{\K}{t, x}$-multilinear maps
\begin{align*}
    \underbrace{\Omega^1_t \times \cdots \times \Omega^1_t}_{r \text{ times}} \times \underbrace{\tdvf \times \cdots \times \tdvf}_{s \text{ times}} \to \powseries{\K}{t, x}.
\end{align*}
As with formal isotopies, time-dependent formal tensor fields may be decorated with a subscript $t$ ($X_t$, $\omega_t$, etc.) to serve as a reminder that their components are power series in both $t$ and the $x^i$'s.

By tensoring with the identity maps $\tensor[r, s] \to \tensor[r, s]$, the inclusion $\imath \colon \powseries{\K}{x} \hookrightarrow \powseries{\K}{t, x}$ induces inclusions $\imath \otimes \id \colon \tensor[r, s] \hookrightarrow \tensor[r, s]_t$ of ``time-independent'' formal tensor fields into time-dependent formal tensor fields. For $(r, s) = (1, 0)$, this simply means that every derivation $\powseries{\K}{x} \to \powseries{\K}{x}$ can be extended to a derivation $\powseries{\K}{t, x} \to \powseries{\K}{t, x}$ by declaring it to vanish on $t$. For $(r, s) = (0, 1)$, this means that every $\powseries{\K}{x}$-linear map $\vf \to \powseries{\K}{x}$ can be post-composed with $\imath$ to obtain an element of $\Hom_{\powseries{\K}{x}}(\vf, \powseries{\K}{t, x}) \cong \Hom_{\powseries{\K}{t, x}}(\tdvf, \powseries{\K}{t, x}) \cong \Omega^1_t$ (using the universal property of extension of scalars). In a basis, this inclusion means: write out any $A \in \tensor[r, s]$ in terms of tensor products of $\coordvf{x^i}$'s and $dx^i$'s, and simply regard each component as a power series in $\powseries{\K}{t, x}$ instead of $\powseries{\K}{x}$.

Conversely, the image of a time-dependent formal tensor field $A_t \in \tensor[r, s]_t$ under the projection $\eval_0 \otimes \id \colon \tensor[r, s]_t \to \tensor[r, s]$ is called its \emph{$t = 0$ time slice} and denoted by $A_0$. Regarding $A_t$ as a $\powseries{\K}{t, x}$-multilinear map
\begin{align*}
    A_t \colon \underbrace{\Omega^1_t \times \cdots \times \Omega^1_t}_{r \text{ times}} \times \underbrace{\tdvf \times \cdots \times \tdvf}_{s \text{ times}} \to \powseries{\K}{t, x},
\end{align*}
this corresponds to pre-composing each factor in the domain with one of the inclusions $\vf \hookrightarrow \tdvf$ or $\Omega^1 \hookrightarrow \Omega^1_t$ just described, then post-composing with $\eval_0$ to obtain the $\powseries{\K}{x}$-multilinear map
\begin{align*}
    A_0 \colon \underbrace{\Omega^1 \times \cdots \times \Omega^1}_{r \text{ times}} \times \underbrace{\vf \times \cdots \times \vf}_{s \text{ times}} \to \powseries{\K}{x}.
\end{align*}
Equivalently, write out $A_t$ in the $\powseries{\K}{t, x}$-basis of tensor products of $\coordvf{x^i}$'s and $dx^i$'s, then evaluate each of its component power series at $t = 0$; this is $A_0 \in \tensor[r, s]$.

The identification of $\tdvf$ with a Lie subalgebra of the set of all derivations of $\powseries{\K}{t, x}$ implies that, for all $X_t \in \tdvf$, the formal Lie derivative $\lie_{X_t}$ (regarded as a formal tensor derivation in one dimension higher) preserves the set of time-dependent formal tensor fields on $\K^n$. In fact, more is true: the Lie bracket $\f{dX_t}{dt} := \big[\coordvf{t}, X_t\big]$ is also in $\tdvf$ even though $\coordvf{t} \notin \tdvf$, so the \emph{time derivative} operator
\begin{align*}
    \f{d}{dt} := \lie_{\coordvf{t}}
\end{align*}
also preserves $\tensor[r, s]_t$. We obtain a split short exact sequence of $\powseries{\K}{x}$-modules
\begin{align*}
    0 \to \tensor[r, s] \xhookrightarrow{\imath \otimes \id} \tensor[r, s]_t \xrightarrow{\f{d}{dt}} \tensor[r, s]_t \to 0,
\end{align*}
which is to say that $A_t \in \tensor[r, s]_t$ is time-independent if and only if $\f{dA_t}{dt} = 0$, if and only if its component power series are all time-independent. When this is the case, we may identify $A_t$ with its $t = 0$ time slice.

Regarding a formal isotopy $\varphi_t \colon \powseries{\K}{t, x} \to \powseries{\K}{t, x}$ of $\K^n$ as a formal diffeomorphism of $\K^{n + 1}$, the pullback $\varphi_t^*$ of formal tensor fields on $\K^{n + 1}$ preserves the set of time-dependent formal tensor fields on $\K^n$. (For example, regarding time-dependent formal vector fields $X_t \in \tdvf$ as derivations of $\powseries{\K}{t, x}$ that vanish on $t$, we have $(\varphi_t^*X_t)(t) = \varphi_t\big(X_t(t)\big) = 0$, so $\varphi_t^*X_t \in \tdvf$.) Given $A_t \in \tensor[r, s]_t$, let us temporarily denote the $t = 0$ time slice of the pullback $\varphi_t^*A_t$ by $(\varphi_t^*A_t)|_{t = 0}$. As one might expect, this coincides with the pullback of the time-independent formal tensor field $A_0$ along the formal diffeomorphism $\varphi_0 \colon \powseries{\K}{x} \to \powseries{\K}{x}$ of $\K^n$ (Lemma \ref{lem:time-zero-slice-composition}):
\begin{align}\label{eq:time-0-eval}
    (\varphi_t^*A_t)|_{t = 0} = \varphi_0^*A_0.
\end{align}
Section \ref{subsec:fixed-time-evaluation} will be dedicated to formulating definitions to make sense of this equality for non-zero values of $t$, provided that $A_t$ and $\varphi_t$ are ``evaluative'' (Proposition \ref{prop:evaluative-properties}).

\subsection{Formal flows}\label{subsec:formal-flows}

Just as smooth vector fields generate flows, so too do their formal versions. Of course, existence and uniqueness of formal solutions to ordinary differential equations is neither new to state nor difficult to prove: one can simply regard a formal vector field $X$ as a tuple of power series and manually find the power series solution to $\dot{x}(t) = X\big(x(t)\big)$ with initial value prescribed at $t = 0$ by comparing coefficients on both sides. However, we could not locate a formulation of this theorem in the literature that both regards formal flows as acting on formal functions and also considers the resulting action on formal tensor fields. Moreover, formal flows of time-dependent vector fields seem less commonly discussed, but will be essential for our version of Moser's trick. As such, we supply a statement and a more conceptual proof.

\begin{theorem}\label{thm:flow-existence-uniqueness}
    For every $X_t \in \tdvf$, there exists a unique formal isotopy $\varphi_t \colon \powseries{\K}{t, x} \to \powseries{\K}{t, x}$, called the \textbf{flow} of $X_t$, satisfying the following two properties:
    \begin{enumerate}
        \item[\namedlabel{(IC)}{itm:initial-condition}] Initial condition: its $t = 0$ time slice $\varphi_0 \colon \powseries{\K}{x} \to \powseries{\K}{x}$ is the identity map.
        \item[\namedlabel{(FP)}{itm:flow-property}] Flow property: for all $f \in \powseries{\K}{x}$, we have
        \begin{align*}
            \coordvf{t} \big(\varphi_t(f)\big) = \varphi_t\big(X_t(f)\big).
        \end{align*}
    \end{enumerate}
\end{theorem}

\begin{remark}\label{rmk:flow-to-vf}
    Conversely, every formal isotopy $\varphi_t$ such that $\varphi_0 = \id_{\powseries{\K}{x}}$ is generated by a unique time-dependent formal vector field $X_t \in \tdvf$, but this is easy to see: the formula
    \begin{align*}
        X_t \colon \powseries{\K}{t, x} \to \powseries{\K}{t, x}, \qquad X_t(f) := \varphi_t^{-1}\left(\coordvf{t}\big(\varphi_t(f)\big)\right) - \f{\p f}{\p t}
    \end{align*}
    defines a derivation of $\powseries{\K}{t, x}$ that vanishes on $t$, and the flow property holds for all $f \in \powseries{\K}{x}$ by construction. (The significance of the extra $-\f{\p f}{\p t}$ term will be explained by Proposition \ref{prop:tensor-flow-property}.)
\end{remark}

\begin{proof}
    \textbf{Step 1: uniqueness.} Given formal isotopies $\varphi_t$, $\psi_t$ satisfying \ref{itm:initial-condition} and \ref{itm:flow-property}, observe that
    \begin{align*}
        S_m := \{f \in \powseries{\K}{t, x} \mid t^m \text{ divides }\varphi_t(f) - \psi_t(f)\} \qquad (m \in \N)
    \end{align*}
    is a $\K$-subalgebra of $\powseries{\K}{t, x}$ by linearity and the identity $\varphi_t(fg) - \psi_t(fg) = \varphi_t(f)\big[\varphi_t(g) - \psi_t(g)\big] + \psi_t(g)\big[\varphi_t(f) - \psi_t(f)\big]$. Moreover, $\bigcap_m S_m$ is exactly the set on which $\varphi_t$ and $\psi_t$ agree. We will prove by induction that $S_m = \powseries{\K}{t, x}$ for each $m$, hence $\varphi_t = \psi_t$.

    The base case $m = 0$ is vacuously true. For the inductive step, assume $m \in \N$ is such that $S_m = \powseries{\K}{t, x}$, so that in particular $X_t(x^i) \in S_m$ for each $i$. We claim that $S_{m + 1}$ contains $t$ and each $x^i$; thus $S_{m + 1}$ contains the polynomial subalgebra $\K[t, x] := \K[t, x^1, \ldots, x^n]$; linearity and continuity then implies that $S_{m + 1} = \powseries{\K}{t, x}$. Since formal isotopies fix $t$, we have $t \in S_{m + 1}$. By \ref{itm:initial-condition}, each of the power series $\Delta^i(t, x) := \varphi_t(x^i) - \psi_t(x^i)$ satisfy $\Delta^i(0, x) = 0$. By \ref{itm:flow-property},
    \begin{align*}
        \f{\p \Delta^i}{\p t} = \varphi_t\big(X_t(x^i)\big) - \psi_t\big(X_t(x^i)\big),
    \end{align*}
    so $t^m$ divides $\f{\p \Delta^i}{\p t}$. Combining these two facts, we see that $t^{m + 1}$ divides $\Delta^i$, hence $x^i \in S_{m + 1}$.

    \textbf{Step 2: existence for time-independent vector fields $X$.} We claim that the flow of $X$ is given by the formal exponential
    \begin{align*}
        e^{tX} \colon \powseries{\K}{t, x} \to \powseries{\K}{t, x}, \qquad e^{tX}(f) := \sum_{k = 0}^\infty \f{t^k}{k!}X^{\circ k}(f),
    \end{align*}
    where $X^{\circ k} := X \circ \cdots \circ X$ denotes the $k$-fold iteration of $X$ (the identity for $k = 0$). Observe that:
    \begin{itemize}
        \item Each term in the infinite sum is well-defined ($k!$ is invertible as $\K$ is a field of characteristic zero) and the sum converges in the $\maxideal_{t, x}$-adic topology as $t^kX^{\circ k}(f) \in \maxideal_{t, x}^k$ for each $k$.
        \item Since each $X^{\circ k}$ is $\K$-linear, so is $e^{tX}$. It is clearly unital, and the higher-order Leibniz rule $X^{\circ k}(fg) = \sum_{i = 0}^k \binom{k}{i}X^{\circ i}(f)X^{\circ (k - i)}(g)$ implies that
        \begin{align*}
            e^{tX}(fg) & = \sum_{k = 0}^\infty \sum_{i = 0}^k\f{t^k}{i!(k - i)!}X^{\circ i}(f)X^{\circ(k - i)}(g) \\
            & = \sum_{k = 0}^\infty \sum_{\substack{i, j \geq 0 \\ i + j = k}} \f{t^{i + j}}{i!j!}X^{\circ i}(f)X^{\circ j}(g) \\
            & = \left(\sum_{i = 0}^\infty \f{t^i}{i!}X^{\circ i}(f)\right)\left(\sum_{j = 0}^\infty \f{t^j}{j!}X^{\circ j}(g)\right) \\
            & = e^{tX}(f)e^{tX}(g),
        \end{align*}
        so $e^{tX}$ is a $\K$-algebra homomorphism.
        \item Using the identity $\sum_{i = 0}^k \f{(-1)^i}{i!(k - i)!} = 0$ for $k > 0$, we compute
        \begin{align*}
            e^{t(-X)}\big(e^{tX}(f)\big) & = \sum_{i = 0}^\infty \sum_{j = 0}^\infty \f{t^{i + j}}{i!j!}(-1)^iX^{\circ (i + j)}(f) = \sum_{k = 0}^\infty t^kX^{\circ k}(f) \sum_{i = 0}^k \f{(-1)^i}{i!(k - i)!} = f.
        \end{align*}
        Reversing the roles of $X$ and $-X$, we see that $e^{tX}$ is inverted by $e^{t(-X)}$.
        \item Since $X^{\circ k}(t) = 0$ for all $k > 0$, we have $e^{tX}(t) = t$, so $e^{tX}$ is a formal isotopy of $\K^n$.
        \item Setting $t = 0$ in the exponential formula immediately shows that $e^{tX}$ satisfies the initial condition \ref{itm:initial-condition}. For the flow property \ref{itm:flow-property}, fix $f \in \powseries{\K}{x}$. Continuity of $\coordvf{t}$ yields
        \begin{align*}
            \coordvf{t} \big(e^{tX}(f)\big) & = \sum_{k = 0}^\infty \coordvf{t}\left(\f{t^k}{k!}X^{\circ k}(f)\right) = \sum_{k = 0}^\infty \f{k}{k!}t^{k - 1}X^{\circ k}(f) + \sum_{k = 0}^\infty \f{t^k}{k!} \coordvf{t}\big(X^{\circ k}(f)\big).
        \end{align*}
        The second sum on the right-hand side vanishes since $X$ and $f$ are time-independent: $\big[\coordvf{t}, X\big] = 0$, hence $\coordvf{t}\big(X^{\circ k}(f)\big) = X^{\circ k}\big(\f{\p f}{\p t}\big) = 0$. Re-indexing the first sum then yields
        \begin{align*}
            \coordvf{t}\big(e^{tX}(f)\big) & = \sum_{k = 0}^\infty \f{t^k}{k!}X^{\circ (k + 1)}(f) = e^{tX}\big(X(f)\big).
        \end{align*}
    \end{itemize}

    \textbf{Step 3: existence for time-dependent vector fields $X_t$.} We mimic the standard procedure for converting a non-autonomous system of $n$ ordinary differential equations into an autonomous system of $n + 1$ equations. That is, temporarily introduce an additional formal variable $s$ and let $X_s$ denote the same vector field as $X_t$, but with every instance of $t$ replaced with $s$. Then $\widetilde{X} := \coordvf{s} + X_s$ is a derivation of $\powseries{\K}{s, x}$---a time-\emph{independent} formal vector field on $\K^{n + 1}$---which has a unique flow $\Phi_t = e^{t\widetilde{X}} \colon \powseries{\K}{t, s, x} \to \powseries{\K}{t, s, x}$ by Steps 1 and 2. We claim that the desired flow of $X_t$ is given by the composition
    \begin{align*}
        \varphi_t \colon \powseries{\K}{t, x} \hookrightarrow \powseries{\K}{t, s, x} \xrightarrow{\Phi_t} \powseries{\K}{t, s, x} \xrightarrow{\text{set } s = 0} \powseries{\K}{t, x}.
    \end{align*}
    Each map in this composition is a $\K$-algebra homomorphism that fixes $t$, so the same is true of $\varphi_t$. To see that $\varphi_t$ is invertible (hence a formal isotopy), observe that its $i$\textsuperscript{th} component is
    \begin{align*}
        \Phi_t(x^i)\big|_{s = 0} = e^{t\widetilde{X}}(x^i)\big|_{s = 0} = x^i + t\widetilde{X}(x^i)\big|_{s = 0} + \f{t^2}{2!}\widetilde{X}^{\circ 2}(x^i)\big|_{s = 0} + \cdots,
    \end{align*}
    from which it follows that $D_x\varphi(0, 0)$ is the identity matrix.

    To show that $\varphi_t$ satisfies the initial condition \ref{itm:initial-condition}, observe that $\varphi_0$ is given by following the top arrows in the diagram below; it is equal to the identity on $\powseries{\K}{x}$ because the diagram commutes:
    \begin{equation*}
        \begin{tikzcd}
            & \powseries{\K}{t, x} \arrow[rd, hookrightarrow] \arrow[rrr, "\varphi_t"] & & & \powseries{\K}{t, x} \arrow[rd, swap, "t = 0"] \\
            \powseries{\K}{x} \arrow[ru, hookrightarrow] \arrow[rd, hookrightarrow] & & \powseries{\K}{t, s, x} \arrow[r, "\Phi_t"] & \powseries{\K}{t, s, x} \arrow[ru, swap, "s = 0"] \arrow[rd, "t = 0"] & & \powseries{\K}{x} \\
            & \powseries{\K}{s, x} \arrow[ru, hookrightarrow] \arrow[rrr, "\id"] & & & \powseries{\K}{s, x} \arrow[ru, "s = 0"]
        \end{tikzcd}
    \end{equation*}
    (The bottom trapezoid commuting amounts to the fact that $\Phi_t$ satisfies \ref{itm:initial-condition} on $\powseries{\K}{s, x}$.)

    To show that $\varphi_t$ satisfies the flow property \ref{itm:flow-property}, first note that $\widetilde{X}(s) = 1$ and $\widetilde{X}^{\circ k}(s) = 0$ for $k > 1$, hence
    \begin{align*}
        \Phi_t(s) = \sum_{k = 0}^\infty \f{t^k}{k!}\widetilde{X}^{\circ k}(s) = s + t.
    \end{align*}
    Observe that $S := \{f \in \powseries{\K}{x} \mid s - t \text{ divides } \widetilde{X}(f) - X_t(f)\}$ is a $\K$-subalgebra of $\powseries{\K}{x}$ by linearity and the Leibniz rule. Moreover, $x^i \in S$ for each $i$ as the terms of $\widetilde{X}(x^i) - X_t(x^i)$ can be grouped into pairs that all have a factor of $s^k - t^k = (s - t)(s^{k - 1} + s^{k - 2}t + \cdots + st^{k - 2} + t^{k - 1})$. Thus $S$ contains the polynomial algebra $\K[x^1, \ldots, x^n]$; linearity and continuity then implies that $S = \powseries{\K}{x}$. Given $f \in \powseries{\K}{x}$, writing $\widetilde{X}(f) - X_t(f) = (s - t)g$ for some $g \in \powseries{\K}{t, s, x}$ yields
    \begin{align*}
        \Phi_t\big(\widetilde{X}(f) - X_t(f)\big) = \Phi_t(s - t)\Phi_t(g) = s\Phi_t(g).
    \end{align*}
    Since $\coordvf{t}$ commutes with setting $s = 0$ and $\Phi_t$ satisfies the flow property on $\powseries{\K}{s, x}$ (but with the vector field $\widetilde{X}$), it follows that
    \begin{align*}
        \coordvf{t}\big(\varphi_t(f)\big) & = \coordvf{t}\left(\Phi_t(f)\big|_{s = 0}\right) \\
        & = \left[\coordvf{t}\big(\Phi_t(f)\big)\right]\bigg|_{s = 0} \\
        & = \Phi_t\big(\widetilde{X}(f)\big)\big|_{s = 0} \\
        & = \big[\Phi_t\big(X_t(f)\big) + s\Phi_t(g)\big]\big|_{s = 0} \\
        & = \varphi_t\big(X_t(f)\big).
    \end{align*}
\end{proof}

The flow property for time-independent formal functions generalizes to a flow property for all formal tensor fields (time-independent or dependent), yielding a formal version of Equation \eqref{eq:pullback-rate-of-change}:

\begin{proposition}\label{prop:tensor-flow-property}
    If $\varphi_t$ is any formal isotopy satisfying the flow property \ref{itm:flow-property} for $X_t$, then
    \begin{align*}
        \f{d}{dt}(\varphi_t^*A_t) = \varphi_t^*\left(\f{dA_t}{dt} + \lie_{X_t} A_t\right), \qquad \forall A_t \in \tensor[r, s]_t.
    \end{align*}
\end{proposition}

\begin{proof}
    As pullbacks and Lie derivatives are compatible (Equation \eqref{eq:pullback-lie-derivative}), we have
    \begin{align*}
        \varphi_t^*\left(\f{dA_t}{dt} + \lie_{X_t} A_t\right) = \varphi_t^*\left(\lie_{\coordvf{t}}A_t + \lie_{X_t}A_t\right) = \lie_{\varphi_t^*\left(\coordvf{t} + X_t\right)}(\varphi_t^*A_t).
    \end{align*}
    Thus it suffices to show that $\coordvf{t} = \varphi_t^*\left(\coordvf{t} + X_t\right)$, or equivalently, that $\coordvf{t} \circ \varphi_t = \varphi_t \circ \left(\coordvf{t} + X_t\right)$ as operators on $\powseries{\K}{t, x}$. Given $f \in \powseries{\K}{t, x}$, write $f(t, x) = \sum_{k = 0}^\infty t^kf_k(x)$ where each $f_k(x) \in \powseries{\K}{x}$ does not depend on $t$. As $\coordvf{t}$ and $\varphi_t$ are continuous and $\varphi_t$ fixes $t$, we have
    \begin{align*}
        \coordvf{t} \big(\varphi_t(f)\big) & = \sum_k \coordvf{t} \big(t^k\varphi_t(f_k)\big) = \sum_k kt^{k - 1}\varphi_t(f_k) + \sum_k t^k\coordvf{t} \big(\varphi_t(f_k)\big).
    \end{align*}
    The first sum is equal to $\varphi_t\left(\f{\p f}{\p t}\right)$. For the second sum, apply \ref{itm:flow-property} to each $f_k$:
    \begin{align*}
        \coordvf{t} \big(\varphi_t(f)\big) & = \varphi_t\left(\f{\p f}{\p t}\right) + \sum_k t^k\varphi_t\big(X_t(f_k)\big) = \varphi_t\bigg(\f{\p f}{\p t} + \sum_k t^kX_t(f_k)\bigg) = \varphi_t\left(\f{\p f}{\p t} + X_t(f)\right),
    \end{align*}
    the last step using the Leibniz rule, the fact that $X_t(t^k) = 0$ for all $k \in \N$, and continuity of $X_t$.
\end{proof}

\begin{corollary}
    If $\varphi_t$ is the flow of $X_t$, then $\varphi_t^{-1}$ is the flow of $-\varphi_t^*X_t$.
\end{corollary}

\begin{proof}
    By Lemma \ref{lem:time-zero-slice-composition} and Theorem \ref{thm:flow-existence-uniqueness}\ref{itm:initial-condition}, we have $(\varphi_t^{-1})_0 = (\varphi_0)^{-1} = \id_{\powseries{\K}{x}}$. By Remark \ref{rmk:flow-to-vf}, $\varphi_t^{-1}$ is the flow of the time-dependent formal vector field
    \begin{align*}
        Y_t := \varphi_t \circ \coordvf{t} \circ \varphi_t^{-1} - \coordvf{t} = \varphi_t^*\left(\coordvf{t}\right) - \coordvf{t}.
    \end{align*}
    The proof of Proposition \ref{prop:tensor-flow-property} showed that $\coordvf{t} = \varphi_t^*\left(\coordvf{t} + X_t\right) = \varphi_t^*\left(\coordvf{t}\right) + \varphi_t^*X_t$, so $Y_t = -\varphi_t^*X_t$.
\end{proof}
\begin{remark}\label{rmk:exponential-recursion}
    We recount an argument that explains where the exponential formula comes from, as similar strategies will be employed to prove later results. If $\varphi_t$ is any formal isotopy satisfying \ref{itm:initial-condition} and \ref{itm:flow-property} for a time-independent formal vector field $X$, then by the $(r, s) = (1, 0)$ case of Proposition \ref{prop:tensor-flow-property},
    \begin{align*}
        \f{d}{dt} (\varphi_t^*X) = \varphi_t^*\left(\f{dX}{dt} + [X, X]\right) = 0.
    \end{align*}
    Thus $\varphi_t^*X = (\varphi_t^*X)|_{t = 0} = X$ by Equation \eqref{eq:time-0-eval} and the initial condition. As $\varphi_t$ is continuous and fixes $t$, it suffices to determine how it acts on the subalgebra $\powseries{\K}{x} \subset \powseries{\K}{t, x}$. Given $f \in \powseries{\K}{x}$, write $\varphi_t(f) = \sum_{k = 0}^\infty t^kf_k$ where each $f_k \in \powseries{\K}{x}$ does not depend on $t$. Then the left-hand side of \ref{itm:flow-property} is
    \begin{align*}
        \coordvf{t} \big(\varphi_t(f)\big) = \sum_{k = 0}^\infty \coordvf{t}\big(t^kf_k\big) = \sum_{k = 0}^\infty (k + 1)t^kf_{k + 1},
    \end{align*}
    while the right-hand side is
    \begin{align*}
        \varphi_t\big(X(f)\big) & = (\varphi_t^*X)(\varphi_t^*f) = X\big(\varphi_t(f)\big) = \sum_{k = 0}^\infty t^kX(f_k).
    \end{align*}
    For \ref{itm:flow-property} to hold, the coefficient power series of $t^k$ on both sides must be equal for all $k$, from which we deduce the recursion relation $(k + 1)f_{k + 1} = X(f_k)$. The base case $f_0 = \varphi_0(f) = f$ is provided by \ref{itm:initial-condition}; it follows by induction that $f_k = \f{1}{k!}X^{\circ k}(f)$, hence $\varphi_t(f) = e^{tX}(f)$.
\end{remark}

\subsection{Fixed-time evaluation}\label{subsec:fixed-time-evaluation}

We now revisit the fundamental question posed in Section \ref{subsec:main-results}---or rather, its formal analogue:
\begin{center}
    \emph{Given formal tensor fields $A_0, A_1 \in \tensor[r, s]$, when does there exist\\a formal diffeomorphism $\varphi \colon \powseries{\K}{x} \to \powseries{\K}{x}$ such that $\varphi^*A_1 = A_0$?}
\end{center}
To emulate Moser's trick, one would like to view $A_0$ and $A_1$ as time slices at $t = 0$ and $t = 1$ of a one-parameter family of formal tensor fields $A_t \colon \K \to \tensor[r, s]$; ditto for $\varphi$. However, Moser-type arguments involve time derivatives such as $\f{dA_t}{dt}$ and $\f{d}{dt}(\varphi_t^*A_t)$, and it is unclear how to make sense of these for arbitrary functions $\K \to \tensor[r, s]$ over an arbitrary field of characteristic zero. On the other hand, the work of Section \ref{subsec:time-dependence} gives meaning to these expressions if $A_t \in \tensor[r, s]_t$ is a time-dependent formal tensor field and $\varphi_t \colon \powseries{\K}{t, x} \to \powseries{\K}{t, x}$ is a formal isotopy, and the work of Section \ref{subsec:formal-flows} allows us to say that the equation $\f{d}{dt}(\varphi_t^*A_t) = 0$ is equivalent to the equation
\begin{align*}
    \f{dA_t}{dt} + \lie_{X_t} A_t = 0, \quad \text{where $X_t \in \tdvf$ generates $\varphi_t$ (assuming that $\varphi_0 = \id_{\powseries{\K}{x}}$).}
\end{align*}
But then comes another issue, which we have already alluded to: while time slices of $A_t$, $\varphi_t$, and $X_t$ at $t = 0$ are perfectly well-defined, replacing the formal variable $t$ with any non-zero $\tau \in \K$ immediately leads to complications, which we illustrate with some examples in dimension $n = 1$.

\begin{example}\label{ex:d/dx-non-evaluative}
    Even if substituting $t = \tau$ into a formal isotopy yields well-defined power series in $x$, there are two distinct ways in which the result can fail to define a formal diffeomorphism. One way is illustrated by the flow of the formal vector field $\coordvf{x}$, which is given by $\varphi_t(x) = x + t$: replacing $t$ with any $\tau \neq 0$ yields a power series with non-zero constant term, which cannot correspond to a $\K$-algebra endomorphism of $\powseries{\K}{x}$ via \eqref{eq:K[[x]]-univ-prop} at all. Another way is illustrated by the formal isotopy given by $\psi_t(x) = x(1 - t)$: substituting $t = 1$ yields the zero power series, which defines a $\K$-algebra endomorphism of $\powseries{\K}{x}$ that is not invertible. Relatedly, since $\psi_t^{-1}(x) = x(1 - t)^{-1} = x\sum_{k = 0}^\infty t^k$, the pullback of $\coordvf{x}$ along $\psi_t$ is $\sum_{k = 0}^\infty t^k\coordvf{x}$, which is ill-defined at $t = 1$.
\end{example}

\begin{example}\label{ex:euler-non-evaluative}
    Regarding the Euler vector field $x\coordvf{x}$ as a formal vector field, its flow is given by
    \begin{align}\label{eq:euler-flow}
        \varphi_t(x) = \sum_{k = 0}^\infty \f{t^k}{k!}x.
    \end{align}
    One might try to make sense of ``substituting $t = 1$ into $\varphi_t$'' by declaring $\varphi_1$ to correspond via \eqref{eq:K[[x]]-univ-prop} to the reasonable-looking expression $\sum_{k = 0}^\infty \f{1}{k!}x$. However, \emph{this infinite sum does not converge in the $\maxideal$-adic topology on $\powseries{\K}{x}$}, as $\Theta(\f{1}{k!}x) = 1$ for all $k$. If the base field comes with a topology such that $\sum_k \f{1}{k!}$ converges to a unique $e \in \K$ (e.g., $\K$ contains the \emph{real number} $e = 2.718\ldots$), then one could declare $\varphi_1$ to be the formal diffeomorphism given by the power series $\varphi_1(x) := ex \in \powseries{\K}{x}$. We will not do this for both practical and philosophical reasons.

    Practically speaking, $\K$ need not be topologized; even if it is, there may be power series that diverge for \emph{all} non-zero $\tau \in \K$. (For example, Euler noted how the series $\sum_{k = 1}^\infty (k - 1)!x^k$, which diverges for all non-zero $x \in \R$, appears when studying the system of differential equations $\dot{x} = x^2$ and $\dot{y} = y - x$ \cite[\S24]{arnold_geometrical_1988}.) Philosophically speaking, using a topology on $\K$ is incongruous with our approach: determining whether an infinite sum converges is generally a difficult task, and the entire \textit{raison d'\^{e}tre} for working with formal power series is to blithely ignore such concerns. Indeed, under the identification of sets $\powseries{\K}{x} \cong \K^{\N^n}$ associating $\sum_\alpha c_\alpha x^\alpha \in \powseries{\K}{x}$ to the function $\N^n \to \K$, $\alpha \mapsto c_\alpha$, the $\maxideal$-adic topology on $\powseries{\K}{x}$ corresponds to the product topology on $\K^{\N^n}$ when $\K$ is equipped with the \emph{discrete} topology. Thus $\powseries{\K}{x}$ does not see any interesting topology on $\K$ anyways.
\end{example}

While we cannot interpret arbitrary formal power series $f \in \powseries{\K}{t, x}$ as ``functions'' $\K \to \powseries{\K}{x}$, $t \mapsto f(t, x)$ and ``evaluate'' them at non-zero $\tau \in \K$, a simple way to sidestep this issue is by requiring these ``functions'' to be polynomial in $t$, i.e., come from the subalgebra $\powseries{\K}{x}[t] \subset \powseries{\K}{t, x}$. Such polynomials can be regarded as honest-to-goodness functions $\K \to \powseries{\K}{x}$, but this turns out to be too restrictive---for one, non-zero formal vector fields \emph{never} generate $\K$-algebra automorphisms of $\powseries{\K}{x}[t]$ (see Remark \ref{rmk:locally-nilpotent}). An intermediate approach is to restrict attention to those formal power series such that, \emph{for each fixed $\alpha \in \N^n$}, the coefficient of the monomial $x^\alpha$ depends polynomially on $t$ (but the degrees of the polynomials may be unbounded as $\alpha$ varies). This is essentially what we shall do, with some additional conditions for formal isotopies:

\begin{definition}\label{def:evaluative}
    We define \textbf{evaluative} formal tensor fields and formal isotopies as follows. A time-dependent formal function $f(t, x) \in \powseries{\K}{t, x}$ is called evaluative if $f(t, x) \in \powseries{\K[t]}{x}$:
    \begin{align*}
        f(t, x) = \sum_{\alpha \in \N^n} p_\alpha(t)x^\alpha, \text{ where each $p_\alpha(t)$ is a polynomial in $t$ (as opposed to a power series).}
    \end{align*}
    A time-dependent formal tensor field $A_t \in \tensor[r, s]_t$ is called evaluative if all of its component power series are evaluative. A formal isotopy $\varphi_t \colon \powseries{\K}{t, x} \to \powseries{\K}{t, x}$, identified with an $n$-tuple $\varphi(t, x) = \big(\varphi^1(t, x), \ldots, \varphi^n(t, x)\big) \in \powseries{\K}{t, x}^n$ via \eqref{eq:isotopy-characterization}, is called evaluative if the following two conditions hold:
    \begin{enumerate}[label=(\roman*)]
        \item\label{itm:evaluative-isotopy-hom} All $n$ component power series $\varphi^i(t, x)$ are evaluative and also satisfy $\varphi^i(t, 0) = 0$.
        \item\label{itm:evaluative-isotopy-aut} The matrix $D_x\varphi(t, x)$ is invertible in $\Mat_n(\powseries{\K[t]}{x})$. Equivalently, $D_x\varphi(t, 0)$ is invertible in $\Mat_n(\K[t])$, or $\det D_x\varphi(t, 0)$ is a non-zero constant polynomial.
    \end{enumerate}
\end{definition}

\begin{example}\label{ex:quadratic-vector-field}
    For any polynomial $p(t) \in \K[t]$, the time-dependent formal vector field in dimension $n = 1$ given by $X_t = p(t)x^2\coordvf{x}$ is evaluative. Defining $P(t) := \int_0^t p(s)\,ds \in \K[t]$, the flow of $X_t$ corresponds to the evaluative power series (which is not in $\powseries{\K}{x}[t]$ unless $p = 0$)
    \begin{align*}
        \varphi(t, x) := x(1 - P(t)x)^{-1} = \sum_{k = 0}^\infty P(t)^kx^{k + 1}.
    \end{align*}
    Since $\varphi(t, 0) = 0$ and $D_x\varphi(t, 0) = 1$, the flow of $X_t$ is an evaluative formal isotopy.
\end{example}

For any $\tau \in \K$, the map $\K[t] \to \K$ given by evaluation of polynomials at $t = \tau$ extends coefficient-wise to a $\powseries{\K}{x}$-algebra homomorphism
\begin{align*}
    \eval_\tau \colon \powseries{\K[t]}{x} \to \powseries{\K}{x}, \qquad f(t, x) \mapsto f(\tau, x).
\end{align*}
Given an evaluative formal tensor field $A_t$, we denote the time-independent formal tensor field obtained by evaluating each of its components at $t = \tau$ by $A_\tau$ (or $A_t|_{t = \tau}$ if confusion is possible). Given an evaluative formal isotopy $\varphi_t$ with corresponding $n$-tuple $\varphi(t, x) = \big(\varphi^1(t, x), \ldots, \varphi^n(t, x)\big) \in \powseries{\K}{t, x}^n$, condition \ref{itm:evaluative-isotopy-hom} implies that $\varphi(\tau, x) := \big(\varphi^1(\tau, x), \ldots, \varphi^n(\tau, x)\big) \in \powseries{\K}{x}^n$ is well-defined and corresponds to a unique $\K$-algebra endomorphism of $\powseries{\K}{x}$ via \eqref{eq:K[[x]]-univ-prop}, denoted by $\varphi_\tau$ or $\varphi_t|_{t = \tau}$. Condition \ref{itm:evaluative-isotopy-aut} guarantees that $\varphi_\tau$ is a formal diffeomorphism. In particular, $A_t|_{t = 0}$ and $\varphi_t|_{t = 0}$ respectively agree with the time zero slices $A_0$ and $\varphi_0$ as defined in Section \ref{subsec:time-dependence}.

\begin{remark}
    One might consider weakening condition \ref{itm:evaluative-isotopy-aut} by only requiring $\det D_x\varphi(\tau, 0)$ to be a non-zero element of $\K$ for each fixed $\tau \in \K$, as this still implies that $\varphi_\tau$ is a formal diffeomorphism when combined with condition \ref{itm:evaluative-isotopy-hom}. However, even if the polynomial function $\tau \mapsto \det D_x\varphi(\tau, 0)$ is nowhere-vanishing on $\K$, it can have zeroes in a field extension of $\K$, and this can lead to undesirable phenomena. For example, the formal isotopy in dimension $n = 1$ given by $\varphi_t(x) = x(t^2 + 1)$ satisfies condition \ref{itm:evaluative-isotopy-hom} and this weaker version of condition \ref{itm:evaluative-isotopy-aut} over $\R$, but not over $\C$; its inverse does not satisfy condition \ref{itm:evaluative-isotopy-hom} as $\varphi_t^{-1}(x) = x(t^2 + 1)^{-1} = x\sum_{k = 0}^\infty (-1)^kt^{2k}$ is not evaluative.

    For completeness, we also note that there is no benefit to relaxing condition \ref{itm:evaluative-isotopy-hom} by only requiring that $\varphi^i(\tau, 0) = 0$ for each fixed $\tau \in \K$: since $\K$ is an infinite field, this is equivalent to $\varphi^i(t, 0)$ being the zero polynomial.
\end{remark}

A formal isotopy is evaluative if and only if it is the extension of a $\K$-algebra automorphism of $\powseries{\K[t]}{x}$ that fixes $t$, i.e., a $\K[t]$-algebra automorphism of $\powseries{\K[t]}{x}$. A time-dependent formal vector field is evaluative if and only if it is the extension of a derivation of $\powseries{\K[t]}{x}$ that vanishes on $t$, i.e., a $\K[t]$-linear derivation of $\powseries{\K[t]}{x}$. (See Remark \ref{rmk:subspace-vs-adic-topology} below for more details.) In theory, we could repeat the developments of Sections \ref{subsec:formal-objects} and \ref{subsec:time-dependence} with $\powseries{\K[t]}{x}$ in mind, \textit{mutatis mutandis}. We list, without proof, some results about evaluative objects obtained by doing so:

\begin{proposition}\label{prop:evaluative-properties}
    Fix $A_t \in \tensor[r, s]_t$, a formal isotopy $\varphi_t$, and $\tau \in \K$.
    \begin{enumerate}[label=(\roman*), ref=\thetheorem{}(\roman*)]
        \item If $\varphi_t$ is evaluative, then so is $\varphi_t^{-1}$. If $\psi_t$ is also an evaluative formal isotopy, then so is $\psi_t \circ \varphi_t$. Evaluation at $t = \tau$ commutes with inversion and composition:
        \begin{align*}
            (\varphi_t^{-1})|_{t = \tau} = (\varphi_\tau)^{-1} \quad\text{and}\quad(\psi_t \circ \varphi_t)|_{t = \tau} = \psi_\tau \circ \varphi_\tau.
        \end{align*}
        \item\label{itm:evaluative-pullback} If $\varphi_t$ and $A_t$ are both evaluative, then so is $\varphi_t^*A_t$. Evaluation commutes with pullback:
        \begin{align*}
            (\varphi_t^*A_t)|_{t = \tau} = (\varphi_\tau)^*(A_\tau).
        \end{align*}
        \item Tensor products, contractions, Lie derivatives along evaluative formal vector fields, and time derivatives of evaluative formal tensor fields are evaluative. Evaluation commutes with the first three operations (but not with $\f{d}{dt}$).
    \end{enumerate}
\end{proposition}

\begin{remark}\label{rmk:subspace-vs-adic-topology}
    Let $I$ be the ideal of $\powseries{\K[t]}{x}$ generated by the $x^i$'s, and let $J$ be the ideal of $\powseries{\K[t]}{x}$ generated by $t$ and the $x^i$'s.\footnote{Maximal ideals of $\powseries{\K[t]}{x}$ are precisely those generated by an irreducible polynomial $p(t) \in \K[t]$ and the $x^i$'s, so $J$ is maximal while $I$ is not.} The $I$-adic topology on $\powseries{\K[t]}{x}$ is the typical choice for topologizing a power series algebra, while the $J$-adic topology coincides with the subspace topology inherited from the $\maxideal_{t, x}$-adic topology on $\powseries{\K}{t, x}$. The $I$-adic topology is strictly finer than the $J$-adic topology. For example, the sequence $(t^k)$ does not converge in the $I$-adic topology, but it converges to $0$ in the $J$-adic topology. As a consequence, $\eval_\tau \colon \powseries{\K[t]}{x} \to \powseries{\K}{x}$ is continuous when $\powseries{\K[t]}{x}$ is equipped with the $I$-adic topology, but is discontinuous for $\tau \neq 0$ when $\powseries{\K[t]}{x}$ is equipped with the $J$-adic topology (as the sequence $(\tau^k)$ does not converge in the $\maxideal$-adic topology of $\powseries{\K}{x}$).

    For evaluative objects, the difference turns out to not matter much. For example, every $\K[t]$-linear derivation $X$ of $\powseries{\K[t]}{x}$ is automatically continuous with respect to the $I$-adic topology, from which it follows that such derivations form a free $\powseries{\K[t]}{x}$-module with basis $\left\{\coordvf{x^i}\right\}$---the proof is exactly the same as for derivations of $\powseries{\K}{x}$ \cite[IV\S4.6]{bourbaki_algebra_2003}. Expanding $X$ in this basis extends $X$ to a derivation of $\powseries{\K}{t, x}$. This extension is continuous in the $\maxideal_{t, x}$-adic topology, so $X$ is also continuous with respect to the subspace (i.e., $J$-adic) topology.

    Similarly, every $\K[t]$-algebra endomorphism $\varphi$ of $\powseries{\K[t]}{x}$ is automatically continuous with respect to both the $I$-adic and $J$-adic topologies, but we cannot justify this using Remark \ref{rmk:adic-continuity} as $\K[t]$ is not a field. We use a degree argument instead: for each $i$, $1 - tx^i$ is invertible in $\powseries{\K[t]}{x}$ as its constant term (with respect to the $x$'s) is invertible in $\K[t]$. Thus $\varphi(1 - tx^i) = 1 - t\varphi(x^i)$ is also invertible in $\powseries{\K[t]}{x}$. If $p^i(t) \in \K[t]$ denotes the constant term of $\varphi(x^i)$, then the constant term of $1 - t\varphi(x^i)$ is $1 - tp^i(t)$, which is therefore invertible in $\K[t]$. By degree considerations, $p^i(t) = 0$, hence $\varphi(x^i) \in I$, hence $\varphi(I) \subseteq I$. As $\varphi(t) = t$, we also have $\varphi(J) \subseteq J$, so $\varphi$ is both $I$-adically and $J$-adically continuous. See \cite{kim_r-automorphisms_1981} for what can happen if one relaxes $\K[t]$-linearity to $\K$-linearity.
\end{remark}

By identifying a time-independent (hence evaluative) formal tensor field with all of its $t = \tau$ time slices where $\tau \in \K$ (we did so in Section \ref{subsec:time-dependence} for $\tau = 0$ only), we arrive at a formal version of Moser's trick:

\begin{corollary}[Formal Moser trick]\label{cor:formal-moser}
    Let $A_t \in \tensor[r, s]_t$ be an evaluative formal tensor field. If there exists $X_t \in \tdvf$ such that the flow $\varphi_t \colon \powseries{\K}{t, x} \to \powseries{\K}{t, x}$ of $X_t$ is evaluative and
    \begin{align*}
        \f{dA_t}{dt} + \lie_{X_t} A_t = 0,
    \end{align*}
    then for all $\tau \in \K$, the formal tensor fields $A_0, A_\tau \in \tensor[r, s]$ and the formal diffeomorphism $\varphi_\tau \colon \powseries{\K}{x} \to \powseries{\K}{x}$ are related by $\varphi_\tau^*A_\tau = A_0$.
\end{corollary}

\begin{proof}
    We have $\f{d}{dt}(\varphi_t^*A_t) = 0$ by Proposition \ref{prop:tensor-flow-property}, hence $\varphi_\tau^*A_\tau = (\varphi_t^*A_t)|_{t = \tau} = (\varphi_t^*A_t)|_{t = 0} = A_0$ by Proposition \ref{itm:evaluative-pullback}.
\end{proof}

\begin{remark}\label{rmk:formal-cartan}
    Differential forms and multivector fields can also be imported into the formal setting by considering exterior powers of the $\powseries{\K}{x}$-modules $\Omega^1$ and $\vf$. Naturally, one can then proceed to develop the full apparatus of a formal Cartan calculus (exterior differentials, interior products, Lie derivatives, Schouten brackets, etc.), as well as time-dependent and evaluative versions, thus obtaining a formal Moser trick for these as well.
\end{remark}

In general, if the flow of $X_t \in \tdvf$ is evaluative, then $X_t$ is evaluative by Remark \ref{rmk:flow-to-vf} and Proposition \ref{prop:evaluative-properties}. However, an evaluative (or even time-independent) formal vector field need not have an evaluative flow, as demonstrated by Examples \ref{ex:d/dx-non-evaluative} and \ref{ex:euler-non-evaluative}. To employ the formal Moser trick, we thus seek conditions guaranteeing the flow of $X_t$ to be evaluative. Combining the next two results provides a partial answer.

\begin{proposition}\label{prop:evaluative-isotopy-condition}
    Let $\varphi_t$ be a formal isotopy of $\K^n$ with corresponding $n$-tuple written in the form $\varphi(t, x) = \sum_{k = 0}^\infty t^k\varphi_k(x)$ where $\varphi_k(x) \in \powseries{\K}{x}^n$ (c.f. Remark \ref{rmk:isotopy-tuple-equivalence}). If $\Theta(\varphi_k) \geq 2$ for all $k \geq 1$ and $\lim_{k \to \infty} \Theta(\varphi_k) = \infty$, then $\varphi_t$ is evaluative.
\end{proposition}

\begin{remark}[a partial converse]\label{rmk:evaluative-isotopy-converse-counterexample}
    If an $n$-tuple $\varphi(t, x) = \sum_k t^k\varphi_k(x)$ corresponds to an evaluative formal isotopy, then $\Theta(\varphi_k) \geq 1$ for all $k \geq 1$ and $\lim_{k \to \infty} \Theta(\varphi_k) = \infty$. It is not necessarily true that $\Theta(\varphi_k) \geq 2$ for $k \geq 1$, as demonstrated by the evaluative formal isotopy of $\K^2$ corresponding to the $2$-tuple $(x + ty, y)$ (the flow of $y\coordvf{x}$).
\end{remark}

\begin{proposition}\label{prop:evaluative-flow}
    Suppose $X_t \in \tdvf$ is written in the form $X_t = \sum_{k = 0}^\infty t^kX_k$ where $X_k \in \vf$. Identify the flow $\varphi_t$ of $X_t$ with an $n$-tuple $\varphi(t, x) = \sum_{k = 0}^\infty t^k\varphi_k(x)$ where $\varphi_k(x) \in \powseries{\K}{x}^n$. If $\Theta\big(X_k(x^i)\big) \geq k + 2$ for all $k \geq 0$ and $i \in \{1, \ldots, n\}$, then $\Theta(\varphi_k) \geq k + 1$ for all $k \geq 0$. In particular, $\varphi_t$ is an evaluative formal isotopy by Proposition \ref{prop:evaluative-isotopy-condition}.
\end{proposition}

Notice that the Euler vector field $x\coordvf{x}$ from Example \ref{ex:euler-non-evaluative} just barely fails the hypothesis for $k = 0$, but the conclusion drastically fails: the flow \eqref{eq:euler-flow} has $\Theta(\varphi_k) = 1$ for all $k$.

Section \ref{subsec:weights-eval} will be dedicated to proving weighted generalizations of Propositions \ref{prop:evaluative-isotopy-condition} and \ref{prop:evaluative-flow}. As such, we will not prove the former at all, and we will only prove the latter in dimension $n = 1$ to convey the main ideas of the more general proof; as we shall see, the argument for $n > 1$ is analogous but requires significantly more complicated notation.

\begin{proof}[Proof in dimension $n = 1$]
    We proceed by (strong) induction on $k$. By Theorem \ref{thm:flow-existence-uniqueness}\ref{itm:initial-condition} we have $\varphi_0(x) = x$, hence $\Theta(\varphi_0) = 1$, establishing the base case. For the inductive step, assume that $\Theta(\varphi_k) \geq k + 1$ for all $k \leq m$. Theorem \ref{thm:flow-existence-uniqueness}\ref{itm:flow-property} with $f(x) = x$ reads
    \begin{align*}
        \coordvf{t}\big(\varphi_t(x)\big) = \varphi_t\big(X_t(x)\big).
    \end{align*}
    The coefficient of $t^m$ on the left-hand side is $(m + 1)\varphi_{m + 1}(x)$; thus it suffices to determine the corresponding coefficient of $t^m$ on the right-hand side and show that its order is at least $m + 2$. As $\Theta(X_k(x)) \geq k + 2$ for all $k$, we can write $X_k(x) = \sum_{\alpha = k + 2}^\infty c_{k, \alpha}x^\alpha$ for some coefficients $c_{k, \alpha} \in \K$. As $\varphi_t$ is a $\K$-algebra automorphism fixing $t$, we thus have
    \begin{align*}
        \varphi_t\big(X_t(x)\big) = \sum_{k = 0}^\infty \sum_{\alpha = k + 2}^\infty c_{k,\alpha}t^k\varphi_t(x)^\alpha.
    \end{align*}
    The term $\varphi_t(x)^\alpha$ can be expressed as a sum over multi-indices of length $\alpha$:
    \begin{align*}
        \varphi_t(x)^\alpha = \Bigg(\sum_{\ell = 0}^\infty t^\ell\varphi_\ell(x)\Bigg)^\alpha = \sum_{\beta \in \N^\alpha} t^{|\beta|}\prod_{\ell = 1}^\alpha \varphi_{\beta_\ell}(x),
    \end{align*}
    from which it follows that the coefficient of $t^m$ in $\varphi_t(X_t(x))$ is
    \begin{align*}
        \sum_{k = 0}^\infty \sum_{\alpha = k + 2}^\infty \sum_{\substack{\beta \in \N^\alpha \\ |\beta| + k = m}} \left(c_{k, \alpha} \prod_{\ell = 1}^\alpha \varphi_{\beta_\ell}(x)\right).
    \end{align*}
    Focusing on the bracketed expression for each fixed $k$, $\alpha$, and $\beta$, observe that $\beta_\ell \leq |\beta| \leq m$, hence $\Theta(\varphi_{\beta_\ell}) \geq \beta_\ell + 1$ by the induction hypothesis. Thus
    \begin{align*}
        \Theta\left(\prod_{\ell = 1}^\alpha \varphi_{\beta_\ell}\right) = \sum_{\ell = 1}^\alpha \Theta(\varphi_{\beta_\ell}) \geq |\beta| + \alpha \geq |\beta| + k + 2 = m + 2.
    \end{align*}
    As the coefficient of $t^m$ in $\varphi_t\big(X_t(x)\big)$ is a sum of terms which all have order $\geq m + 2$, it follows that $\Theta(\varphi_{m + 1}) \geq m + 2$, completing the induction.
\end{proof}

We conclude this section by drawing some comparisons between vector fields whose flows are evaluative and two well-known types of derivations that generate automorphisms of algebras similar to $\powseries{\K[t]}{x}$: one obtained by swapping which variables are allowed to appear with unbounded degrees ($\powseries{\K}{t}[x]$), and one obtained by swapping the order in which the variables are adjoined ($\powseries{\K}{x}[t]$).

\begin{remark}\label{rmk:poly-flow}
    A time-independent $C^1$ vector field $X$ on $\R^n$ is called \emph{poly-flow} if its time $t$ flow $\varphi_t$ is a polynomial mapping (i.e., the $n$ components of $\varphi_t$ are all polynomials in $x$) wherever defined. Bass--Meisters \cite[Theorem 4.1]{bass_polynomial_1985} proved that this condition implies the following: $X$ is complete and its component functions are polynomials; the degrees of the components of $\varphi_t$ are bounded above as $t \in \R$ varies; the coefficients of these polynomials are analytic functions of $t$. Thus, the flow of a poly-flow vector field corresponds to an automorphism of $\powseries{\R}{t}[x]$. Bass--Meisters also show that $\det D_x\varphi(t, x)$ depends only on $t$, not $x$, hence is equal to $\det D_x\varphi(t, 0)$. This is somewhat orthogonal to our definition of an evaluative formal isotopy, in which we allow the Jacobian determinant to depend on $x$, but not on $t$ when $x = 0$.

    If $\K$ is any field of characteristic zero and $f \colon \K^n \to \K^n$ is a bijective function such that $f$ and $f^{-1}$ are both polynomial mappings, then the chain rule implies that $\det Df(x)$ is an invertible element of $\K[x]$, hence a non-zero constant. The famous \emph{Jacobian conjecture} asks: if $f$ is a polynomial mapping and $\det Df(x)$ is a non-zero constant, must $f$ be bijective with its inverse also a polynomial mapping? For $\K = \R$ and $\K = \C$, Meisters--Olech \cite{meisters_poly-flow_1987} have shown that the Jacobian conjecture holding for a given $f$ is equivalent to the assertion that, for each $a \in \K^n$, the flow $\varphi_t(x; a)$ of the time-independent vector field $X_a := Df(x)^{-1}a$ is polynomial in both $x$ \emph{and} $t$ (where $t$ is a complex time parameter if $\K = \C$). They construct $f^{-1}(a)$ for any given $a = (a^1, \ldots, a^n)$ in a fashion similar to evaluation: denoting the standard basis vectors of $\K^n$ by $e_i$, start at the origin $\alpha_0 = (0, \ldots, 0)$ and successively flow along each $X_{e_i}$ for time $a^i$ to obtain the sequence of points $\alpha_i := \varphi_{a^i}(\alpha_{i - 1}; e_i)$. Then $f^{-1}(a) = \alpha_n$, which is polynomial in $a$ if each flow is polynomial in space and time.

    We speculate that it might be possible to perform their construction over more general fields of characteristic zero via a formal Moser-type argument, with a caveat: the formal flow of $X_a$ being globally polynomial in $x$ and $t$ does not imply that it is evaluative. For example, the polynomial mapping $f(x, y) = (x + y^2, y)$ has polynomial inverse $f^{-1}(u, v) = (u - v^2, v)$; for this $f$, the flow of $X_{(a, b)}$ corresponds to the $2$-tuple $(x + (a - 2by)t - b^2t^2, y + bt)$, which satisfies condition \ref{itm:evaluative-isotopy-aut} of Definition \ref{def:evaluative} but not condition \ref{itm:evaluative-isotopy-hom} (unless $(a, b) = (0, 0)$).
\end{remark}

\begin{remark}\label{rmk:locally-nilpotent}
    A derivation $X$ of an algebra $A$ is called \emph{locally nilpotent} \cite{freudenburg_algebraic_2017} if, for each $a \in A$, there exists $k \in \N$ such that $X^{\circ k}(a) = 0$. Equivalently, extend $X$ to a derivation of $\powseries{A}{t}$ by defining $X(\sum_k a_kt^k) = \sum_k X(a_k)t^k$; then $X$ is locally nilpotent if and only if the algebra automorphism $e^{tX} \colon \powseries{A}{t} \to \powseries{A}{t}$ defined exactly as in Theorem \ref{thm:flow-existence-uniqueness} preserves $A[t]$. However, \emph{the only locally nilpotent derivation of $A = \powseries{\K}{x}$ is the zero derivation}.

    Indeed, suppose $\varphi_t$ is any $\K$-algebra endomorphism of $\powseries{\K}{t, x}$ that fixes $t$ and satisfies \ref{itm:initial-condition} (e.g., the flow of any formal vector field). If $\varphi_t$ preserves $\powseries{\K}{x}[t]$, then $\varphi_t(1 - x^i) = 1 - \varphi_t(x^i)$ is invertible in $\powseries{\K}{x}[t]$ as $1 - x^i$ is invertible in $\powseries{\K}{x}$. Thus $\varphi_t(x^i)$ is a constant polynomial with respect to $t$; in particular, it equals its $t = 0$ time slice, which is $x^i$ by \ref{itm:initial-condition}. Since $\varphi_t$ fixes $t$ and each $x^i$, it restricts to the identity map on the polynomial algebra $\K[t, x]$, which is a dense subset of $\powseries{\K}{t, x}$ in the (Hausdorff) $\maxideal_{t, x}$-adic topology. By continuity, $\varphi_t = \id_{\powseries{\K}{t, x}}$; by Remark \ref{rmk:flow-to-vf}, $\varphi_t$ is the flow of the zero vector field.
\end{remark}

\section{Weightings of \texorpdfstring{$\K^n$}{K\^{}n}}\label{sec:weightings}

\subsection{The algebra of weights}\label{subsec:weights}

We continue to use the shorthand notation $\powseries{\K}{x} := \powseries{\K}{x^1, \ldots, x^n}$ for the algebra of formal power series and $\K[x] := \K[x^1, \ldots, x^n]$ for the subalgebra of polynomials. A \emph{polynomial vector field} on $\K^n$ is a derivation of $\K[x]$. The set $\vf_\poly$ of polynomial vector fields is both a module over $\K[x]$ with a basis given by $\left\{\coordvf{x^1}, \ldots, \coordvf{x^n}\right\}$ and a Lie algebra over $\K$. For example, the poly-flow vector fields discussed in Remark \ref{rmk:poly-flow} are polynomial vector fields. We may naturally regard $\vf_\poly \subset \vf$ as a Lie subalgebra of the formal vector fields, and when $\K = \R$ or $\K = \C$, also as a Lie subalgebra of the smooth vector fields on $\K^n$.

\begin{definition}\label{def:weighting}
    A \textbf{weighting} of $\K^n$ is an $n$-tuple of integers $w = (w_1, w_2, \ldots, w_n)$ such that $1 \leq w_1 \leq w_2 \leq \dots \leq w_n$.
\end{definition}

We associate the $i$\textsuperscript{th} weight $w_i$ with both the formal variable $x^i \in \K[x] \subset \powseries{\K}{x}$ and the coordinate function $x^i \colon \K^n \to \K$. By extension, the \emph{weighted degree} of a monomial $x^\alpha = (x^1)^{\alpha_1} \cdots (x^n)^{\alpha_n}$ with respect to $w$ is $\pair{w, \alpha} := \sum_i w_i\alpha_i$. The ideal of $\powseries{\K}{x}$ generated by monomials $x^\alpha$ such that $\pair{w, \alpha} \geq k$ is denoted by $\mathcal{F}^k$. Note that $\mathcal{F}^k = \maxideal$ for $1 \leq k \leq w_1$, as $\mathcal{F}^1 \subseteq \maxideal \subseteq \mathcal{F}^{w_1} \subseteq \mathcal{F}^1$.

Requiring weights to be listed from smallest to largest is done out of convenience to simplify proofs, e.g., those of Proposition \ref{prop:evaluative-isotopy-weighted-condition} and Lemma \ref{lem:linear-approximation-commutative}. One can of course start with weights in a different order, permute the coordinates of $\K^n$ to sort them, then apply the upcoming results. However, the requirement that weights be strictly positive is not arbitrary---for one, it guarantees that the translation-invariant topology on $\powseries{\K}{x}$ determined by the decreasing filtration by ideals
\begin{align*}
    \powseries{\K}{x} = \mathcal{F}^0 \supseteq \mathcal{F}^1 \supseteq \mathcal{F}^2 \supseteq \cdots
\end{align*}
coincides with the $\maxideal$-adic topology. From this filtration, we obtain the associated grading
\begin{align*}
    \K[x] = \bigoplus_{k = 0}^\infty \mathcal{P}^k,
\end{align*}
where the $\K$-vector space $\mathcal{P}^k$ of \emph{quasi-homogeneous polynomials} of weighted degree $k$ has basis given by $\{x^\alpha \mid \pair{w, \alpha} = k\}$. The adjective ``unweighted'' shall always refer to the \emph{trivial weighting} $w = (1, \ldots, 1)$, from which we recover the $\maxideal$-adic filtration on $\powseries{\K}{x}$ ($\mathcal{F}^k = \maxideal^k$) and the usual grading on $\K[x]$ by polynomial homogeneity. Given $f(x) = \sum_\alpha c_\alpha x^\alpha \in \powseries{\K}{x}$, we define
\begin{align*}
    \Theta_w(f) := \max\{k \mid f \in \mathcal{F}^k\} = \min\{\pair{w, \alpha} \mid c_\alpha \neq 0\},
\end{align*}
with the convention that $\Theta_w(0) := \infty$. Like unweighted Big Theta, this satisfies $\Theta_w(fg) = \Theta_w(f) + \Theta_w(g)$ because $\K$ is an integral domain. For non-trivial weightings, we do \emph{not} extend this ``weighted Big Theta'' notation to $n$-tuples of power series, as different components may have different weights.

\begin{example}
    On $\K^2$ with coordinates $(x, y)$, we may consider $x$ ``quadratic'' and $y$ ``cubic'' by assigning weights $(w_1, w_2) = (2, 3)$. Then the filtration on $\powseries{\K}{x, y}$ starts with $\mathcal{F}^1 = \mathcal{F}^2 = \maxideal$, and
    \begin{align*}
        \mathcal{F}^3 & = \operatorname{span}_{\powseries{\K}{x, y}} \{x^2, y\}, & \mathcal{F}^4 & = \operatorname{span}_{\powseries{\K}{x, y}} \{x^2, xy, y^2\}, & \mathcal{F}^5 & = \operatorname{span}_{\powseries{\K}{x, y}} \{x^3, xy, y^2\}.
    \end{align*}
    For this weighting, there are no non-zero quasi-homogeneous polynomials of weighted degree $1$.
\end{example}

So far, there has been a lot of algebra and not much calculus, so let us compare with analogous definitions in the differentiable world. Recall that a function $f \colon \R^n \to \R$ is called positively homogeneous of degree $s \in \R$ if $f(tx) = t^sf(x)$ for all $x \in \R^n$ and $t > 0$. For a function $f$ that is smooth at $0 \in \R^n$ and positively homogeneous of degree $s$, one has the following dichotomy:
\begin{itemize}
    \item If $s < 0$, then $f$ is the zero function.
    \item If $s \geq 0$, then $s \in \N$ and $f$ is a homogeneous polynomial function of degree $s$.
\end{itemize}
Thus, smooth positively homogeneous functions can be studied purely algebraically. Introducing a weighting $w = (w_1, \ldots, w_n)$, let us call a function $f \colon \R^n \to \R$ \emph{positively quasi-homogeneous} of weighted degree $s \in \R$ if $\kappa_t^*f = t^sf$ for all $t > 0$, where $\kappa_t \colon \R^n \to \R^n$ is the linear map given by
\begin{align}\label{eq:weighted-action}
    \kappa_t(x^1, \ldots, x^n) = (t^{w_1}x^1, \ldots, t^{w_n}x^n).
\end{align}
Grabowski--Rotkiewicz \cite[Lemma 2.1]{grabowski_graded_2012} have shown that the dichotomy above continues to hold when one replaces ``(positively) homogeneous'' with ``(positively) quasi-homogeneous'' and ``degree'' with ``weighted degree''. Thus, smooth positively quasi-homogeneous functions can also be studied through an algebraic lens.

The diffeomorphisms $\kappa_t$ pull back the coordinate vector fields $\coordvf{x^i}$ on $\R^n$ to $\kappa_t^*\coordvf{x^i} = t^{-w_i}\coordvf{x^i}$. Mnemonically, one thinks of $\coordvf{x^i}$ as having weighted degree $-w_i$ ``since $x^i$ is in the denominator''; formally, the filtration on $\powseries{\K}{x}$ induces a filtration on $\vf$ that starts in negative degrees and is compatible with the Lie algebra structure. We denote the filtration index with double brackets:
\begin{align*}
    \vf = \vfFilt{-w_n} \supseteq \vfFilt{-w_n + 1} \supseteq \vfFilt{-w_n + 2} \supseteq \cdots,
\end{align*}
where $\vfFilt{k}$ is the $\powseries{\K}{x}$-submodule of formal vector fields that act on $\powseries{\K}{x}$ by raising filtration degree by $k$. Equivalently, $\vfFilt{k}$ is generated by monomial vector fields $x^\alpha\coordvf{x^i}$ such that $\pair{w, \alpha} - w_i \geq k$. Similarly, the associated grading on $\K[x]$ makes $\vf_\poly$ into a graded Lie algebra
\begin{align}\label{eq:vf-poly-grading}
    \vf_\poly = \bigoplus_{k = -w_n}^\infty \vfGrad{k},
\end{align}
where the $\K$-linear subspace $\vfGrad{k}$ of \emph{quasi-homogeneous vector fields} of weighted degree $k$ acts on $\K[x]$ by raising grading degree by $k$. Equivalently, $\vfGrad{k}$ has a basis given by monomial vector fields $x^\alpha\coordvf{x^i}$ such that $\pair{w, \alpha} - w_i = k$; thus
\begin{align}\label{eq:quasi-homogeneous-vf}
    \vfGrad{k} \cong \bigoplus_{i = 1}^n \mathcal{P}^{w_i + k} \text{ as }\K\text{-vector spaces}.
\end{align}

\begin{remark}\label{rmk:kappa_t-laurent}
    By interpreting the right-hand side of \eqref{eq:weighted-action} as an $n$-tuple of formal power series, one obtains a $\K$-algebra endomorphism $\kappa_t \colon \powseries{\K}{t, x} \to \powseries{\K}{t, x}$ such that $\kappa_t(t) = t$ and $\kappa_t(x^i) = t^{w_i}x^i$. This is \emph{not} a formal isotopy, as its inverse would have to map $x^i \mapsto t^{-w_i}x^i$. But if one extends $\kappa_t$ to act on Laurent series in $t$ and suitably defines pullbacks in this setting, then $\vfGrad{k}$ is the set of polynomial vector fields $X$ satisfying $\kappa_t^*X = t^kX$. (When $\K = \R$, this is equivalent to requiring $\kappa_t^*X = t^kX$ in the smooth sense for all $t > 0$ by Grabowski--Rotkiewicz.)
\end{remark}

\subsection{Weights and evaluation}\label{subsec:weights-eval}

We now state and prove the promised weighted generalizations of Propositions \ref{prop:evaluative-isotopy-condition} and \ref{prop:evaluative-flow}.

\begin{proposition}\label{prop:evaluative-isotopy-weighted-condition}
    Let $\varphi_t$ be a formal isotopy of $\K^n$ with corresponding $n$-tuple written in the form $\varphi(t, x) = \sum_{k = 0}^\infty t^k\varphi_k(x)$ where $\varphi_k(x) = \big(\varphi_k^1(x), \ldots, \varphi_k^n(x)\big) \in \powseries{\K}{x}^n$. If there exists a weighting $w = (w_1, \ldots, w_n)$ of $\K^n$ such that, for all $i \in \{1, \ldots, n\}$, we have
    \begin{align*}
        \Theta_w(\varphi_0^i) \geq w_i, && \Theta_w(\varphi_k^i) \geq w_i + 1 \text{ for all }k \geq 1, && \lim_{k \to \infty} \Theta(\varphi_k) = \infty,
    \end{align*}
    then $\varphi_t$ is evaluative.
\end{proposition}

Taking the trivial weighting $w = (1, \ldots, 1)$ recovers Proposition \ref{prop:evaluative-isotopy-condition} as the condition on $\varphi_0$ reduces to $\Theta(\varphi_0) \geq 1$, which formal isotopies always satisfy (Remark \ref{rmk:isotopy-tuple-equivalence}). Note that the limit condition is stated with \emph{unweighted} Big Theta, but it is equivalent to requiring that $\lim_{k \to \infty} \Theta_w(\varphi_k^i) = \infty$ for each $i$ since the filtration topology defined by $w$ coincides with the $\maxideal$-adic topology; both limits are equivalent to saying that $\lim_{k \to \infty} \varphi_k^i = 0$ for each $i$ in this topology.

\begin{proof}
    By definition, $\Theta(\varphi_k)$ is the smallest $\ell \in \N$ for which there exists $i \in \{1, \ldots, n\}$ and $\alpha \in \N^n$ such that $|\alpha| = \ell$ and the coefficient of $x^\alpha$ in $\varphi_k^i(x)$ is non-zero (unless $\varphi_k(x) = (0, \ldots, 0)$). For any particular $\alpha$, we have $\Theta(\varphi_k) > |\alpha|$ for sufficiently large $k$ since $\Theta(\varphi_k) \to \infty$. Thus the coefficient of $x^\alpha$ in each $\varphi_k^i(x)$ is non-zero for finitely many $k$'s, so the coefficient of $x^\alpha$ in $\varphi^i(t, x) = \sum_k t^k\varphi_k^i(x)$ is actually a polynomial in $t$. As $\alpha$ was arbitrary, each $\varphi^i$ is an evaluative power series. Since $\Theta_w(\varphi_k^i) \geq w_i \geq 1$ for all $k \geq 0$, we have $\varphi_k^i(x) \in \mathcal{F}^{w_i} \subseteq \mathcal{F}^1 = \maxideal$, hence $\varphi^i(t, 0) = \sum_k t^k\varphi_k^i(0) = 0$. Thus condition \ref{itm:evaluative-isotopy-hom} of Definition \ref{def:evaluative} is satisfied, and we are left with checking condition \ref{itm:evaluative-isotopy-aut}.

    We claim that $D\varphi_k(0)$ is a block upper triangular matrix for all $k \geq 0$, and \emph{strictly} block upper triangular for all $k \geq 1$; that is,
    \begin{align}\label{eq:block-diagonal-jacobian}
        D\varphi_k(0) & =
        \begin{pmatrix}
            A_{k,1} & * & \cdots & * \\
            O & \ddots & \ddots & \vdots \\
            \vdots & \ddots & \ddots & * \\
            O & \cdots & O & A_{k,m}
        \end{pmatrix}, \quad \text{where } A_{k, \ell} = O \text{ for } k \geq 1.
    \end{align}
    Here, $m$ denotes the number of \emph{distinct} weights and each $A_{k, \ell}$ is a $d_\ell \times d_\ell$ matrix, where $d_\ell$ is the number of coordinates assigned to the $\ell$\textsuperscript{th} smallest weight value.\footnote{For example, if $w = (1, 4, 5, 5)$, then there are $m = 3$ distinct weights and $(d_1, d_2, d_3) = (1, 1, 2)$.}

    To prove the claim, recall that $\coordvf{x^j}$ decreases filtration degree on $\powseries{\K}{x}$ by $w_j$, so
    \begin{align*}
        \Theta_w\left(\f{\p \varphi_k^i}{\p x^j}\right) \geq \Theta_w(\varphi_k^i) - w_j.
    \end{align*}
    Since we require weights to be sorted from smallest to largest in Definition \ref{def:weighting} and the partitioning of $D\varphi_k(0)$ illustrated in \eqref{eq:block-diagonal-jacobian} is along blocks corresponding to distinct weights, it follows that entries below the main block diagonal are those indexed by $(i, j)$ where $w_i > w_j$. Since $\Theta_w(\varphi_k^i) \geq w_i$ for all $k \geq 0$, we thus have $\Theta_w\Big(\f{\p \varphi_k^i}{\p x^j}\Big) \geq w_i - w_j > 0$, hence $\f{\p \varphi_k^i}{\p x^j}(0) = 0$ below the main block diagonal. Entries within the main block diagonal are indexed by $(i, j)$ with $w_i = w_j$, so if $k \geq 1$, then $\Theta_w\Big(\f{\p \varphi_k^i}{\p x^j}\Big) \geq w_i - w_j + 1 = 1$, hence also $\f{\p \varphi_k^i}{\p x^j}(0) = 0$. This proves the claim.

    Now observe that $D_x\varphi(t, x) = \sum_{k = 0}^\infty t^kD\varphi_k(x)$, so upon setting $x^1 = \cdots = x^n = 0$, we obtain
    \begin{align*}
        D_x\varphi(t, 0) & = \sum_{k = 0}^\infty t^kD\varphi_k(0) = 
        D\varphi_0(0) + \sum_{k = 1}^\infty t^k \cdot \text{(strictly block upper triangular matrices)}.
    \end{align*}
    Consequently $\det D_x\varphi(t, 0) = \det D\varphi_0(0)$, which is a non-zero element of $\K$ since $\varphi_t$ is a formal isotopy (Remark \ref{rmk:isotopy-tuple-equivalence}). Thus condition \ref{itm:evaluative-isotopy-aut} of Definition \ref{def:evaluative} is satisfied.
\end{proof}

\begin{proposition}\label{prop:evaluative-flow-weighted-condition}
    Suppose $X_t \in \tdvf$ is written in the form $X_t = \sum_{k = 0}^\infty t^kX_k$ where $X_k \in \vf$. Identify the flow $\varphi_t$ of $X_t$ with an $n$-tuple $\varphi(t, x) = \sum_{k = 0}^\infty t^k\varphi_k(x)$ where $\varphi_k(x) \in \powseries{\K}{x}^n$. If there exists a weighting $w = (w_1, \ldots, w_n)$ of $\K^n$ such that $X_k \in \vfFilt{k + 1}$ for all $k \geq 0$, then with respect to this weighting, one has $\Theta_w(\varphi_k^i) \geq w_i + k$ for all $i \in \{1, \ldots, n\}$ and $k \geq 0$. In particular, $\varphi_t$ is an evaluative formal isotopy by Proposition \ref{prop:evaluative-isotopy-weighted-condition}.
\end{proposition}

\begin{proof}
    We proceed by (strong) induction on $k$. By Theorem \ref{thm:flow-existence-uniqueness}\ref{itm:initial-condition} we have $\varphi_0^i(x) = x^i$, hence $\Theta_w(\varphi_0^i) = w_i$ for each $i \in \{1, \ldots, n\}$, establishing the base case. For the inductive step, assume that $\Theta_w(\varphi_k^i) \geq w_i + k$ for all $i \in \{1, \ldots, n\}$ and $k \leq m$. The coefficient of $t^m$ on the left-hand side of Theorem \ref{thm:flow-existence-uniqueness}\ref{itm:flow-property} applied to each $x^i$ is $(m + 1)\varphi_{m + 1}^i(x)$; we want to find $\Theta_w$ of the corresponding coefficient of $t^m$ on the right-hand side. Warning: an unglamorous calculation follows. The reader may wish to skip to the final expression \eqref{eq:horrible-expression} for this coefficient.

    As $\Theta_w(x^i) = w_i$ and $X_k$ raises filtration degree by $k + 1$, we have $\Theta_w\big(X_k(x^i)\big) \geq w_i + k + 1$. Thus
    \begin{align*}
        X_k(x^i) & = \sum_{\substack{\alpha \in \N^n \\ \pair{w, \alpha} \geq w_i + k + 1}} c_{k, \alpha}^i x^\alpha, \quad \text{for some coefficients }c_{k, \alpha}^i \in \K.
    \end{align*}
    As $\varphi_t$ is a $\K$-algebra automorphism fixing $t$, we thus have
    \begin{align*}
        \varphi_t\big(X_t(x^i)\big) & = \sum_{k = 0}^\infty t^k\varphi_t\big(X_k(x^i)\big) = \sum_{k = 0}^\infty \sum_{\substack{\alpha \in \N^n \\ \pair{w, \alpha} \geq w_i + k + 1}} c_{k, \alpha}^i t^k\prod_{j = 1}^n \varphi_t(x^j)^{\alpha_j}.
    \end{align*}
    For each $j \in \{1, \ldots, n\}$, we can express $\varphi_t(x^j)^{\alpha_j}$ as a sum over multi-indices of length $\alpha_j$ like in the proof of Proposition \ref{prop:evaluative-flow}:
    \begin{align*}
        \varphi_t(x^j)^{\alpha_j} = \left(\sum_{\ell = 0}^\infty t^\ell \varphi_\ell^j(x)\right)^{\alpha_j} = \sum_{\beta^j \in \N^{\alpha_j}} t^{|\beta^j|}\prod_{\ell = 1}^{\alpha_j} \varphi_{\beta^j_\ell}^j(x).
    \end{align*}
    The $n$ multi-indices $\beta^j \in \N^{\alpha_j}$ can be concatenated into a single multi-index $\beta := (\beta^1, \ldots, \beta^n)$ of total length $|\alpha|$. Conversely, any $\beta \in \N^{|\alpha|} \cong \N^{\alpha_1} \times \cdots \times \N^{\alpha_n}$ splits into $n$ multi-indices $\beta^j \in \N^{\alpha_j}$, where $\beta^j$ is the part of $\beta$ indexed by $L_\alpha^j := \{\ell \in \N \mid \alpha_1 + \cdots + \alpha_{j - 1} < \ell \leq \alpha_1 + \cdots + \alpha_j\}$. Thus
    \begin{align*}
        \prod_{j = 1}^n \varphi_t(x^j)^{\alpha_j} & = \sum_{\beta^1 \in \N^{\alpha_1}} \cdots \sum_{\beta^n \in \N^{\alpha_n}} t^{|\beta^1| + \cdots + |\beta^n|} \prod_{j = 1}^n \prod_{\ell = 1}^{\alpha_j} \varphi_{\beta^j_\ell}^j(x) \\
        & = \sum_{\beta \in \N^{|\alpha|}} t^{|\beta|} \prod_{j = 1}^n \prod_{\ell \in L_\alpha^j} \varphi_{\beta_\ell}^j(x).
    \end{align*}
    It follows that the coefficient of $t^m$ on the right-hand side of Theorem \ref{thm:flow-existence-uniqueness}\ref{itm:flow-property} applied to each $x^i$ is
    \begin{align}\label{eq:horrible-expression}
        \sum_{k = 0}^\infty \sum_{\substack{\alpha \in \N^n \\ \pair{w, \alpha} \geq w_i + k + 1}} \sum_{\substack{\beta \in \N^{|\alpha|} \\ |\beta| + k = m}} c^i_{k, \alpha}\prod_{j = 1}^n \prod_{\ell \in L_\alpha^j} \varphi_{\beta_\ell}^j(x). \tag{$\dagger$}
    \end{align}
    The salient point is that for each fixed $k$, $\alpha$, and $\beta$ appearing in \eqref{eq:horrible-expression}, we have $\beta_\ell \leq |\beta| \leq m$ for each $1 \leq \ell \leq |\alpha|$, hence $\Theta_w(\varphi_{\beta_\ell}^j) \geq w_j + \beta_\ell$ by the induction hypothesis. Thus
    \begin{align*}
        \Theta_w\Bigg(\prod_{j = 1}^n \prod_{\ell \in L_\alpha^j} \varphi_{\beta_\ell}^j\Bigg) = \sum_{j = 1}^n \sum_{\ell \in L_\alpha^j} \Theta_w(\varphi_{\beta_\ell}^j) \geq \sum_{j = 1}^n \sum_{\ell \in L_\alpha^j} (w_j + \beta_\ell).
    \end{align*}
    Observing that
    \begin{align*}
        \sum_{j = 1}^n \sum_{\ell \in L_\alpha^j} w_j & = \sum_{j = 1}^n w_j\alpha_j = \pair{w, \alpha} \geq w_i + k + 1, & \sum_{j = 1}^n \sum_{\ell \in L_\alpha^j} \beta_\ell & = |\beta|,
    \end{align*}
    we see that $\sum_{j = 1}^n \sum_{\ell \in L_\alpha^j} (w_j + \beta_\ell) \geq w_i + k + 1 + |\beta| = w_i + m + 1$. Thus $\Theta_w(\varphi_{m + 1}^i) \geq w_i + m + 1$, completing the induction.
\end{proof}

\begin{example}
    As noted in Remark \ref{rmk:evaluative-isotopy-converse-counterexample}, the flow of $y\coordvf{x}$ is the formal isotopy of $\K^2$ corresponding to the $2$-tuple $(x + ty, y)$. We could not conclude that this is evaluative from either Propositions \ref{prop:evaluative-isotopy-condition} or \ref{prop:evaluative-flow}, but we can from Proposition \ref{prop:evaluative-flow-weighted-condition} using any weighting $(w_1, w_2)$ such that $w_1 < w_2$. (Note that with respect to such a weighting, $y\coordvf{x}$ is exactly the ``weighted non-linear perturbation'' in the vector field \eqref{eq:weighted-euler-like-example} when $k = 1$.) The author does not know the answer to the following question: if a formal isotopy of $\K^n$ is evaluative, does there exist \emph{some} weighting of $\K^n$ with respect to which the hypotheses of Proposition \ref{prop:evaluative-isotopy-weighted-condition} hold, perhaps after permuting coordinates?
\end{example}

Requiring that $X_k \in \vfFilt{k + 1}$ in Proposition \ref{prop:evaluative-flow-weighted-condition} is probably stronger than needed to conclude that the flow of $X_t := \sum_k t^kX_k$ is evaluative via Proposition \ref{prop:evaluative-isotopy-weighted-condition}. For example, \eqref{eq:horrible-expression} stays almost the same with the weaker hypothesis that $X_k \in \vfFilt{1}$ for all $k \geq 0$, except the condition $\pair{w, \alpha} \geq w_i + k + 1$ in the sum over $\alpha \in \N^n$ is replaced with $\pair{w, \alpha} \geq w_i + 1$; this is enough to prove by induction that $\Theta_w(\varphi_k^i) \geq w_i + 1$ for all $k \geq 1$. We conjecture that $\Theta(\varphi_k) \to \infty$ if $\Theta(X_k) \to \infty$ (the Big Theta notation for $X_k$ understood by identifying it with the $n$-tuple of its components), which would allow the use of Proposition \ref{prop:evaluative-isotopy-weighted-condition}. It turns out that the vector fields we will flow along in Section \ref{sec:weighted-lin} satisfy the hypothesis of Proposition \ref{prop:evaluative-flow-weighted-condition} as written, so we content ourselves with the current statement.

\section{Weighted linearization}\label{sec:weighted-lin}

\subsection{Weighted linear approximations}\label{subsec:weighted-lin-approx}

Every formal function $f \in \powseries{\K}{x}$ can be written as a convergent infinite sum of homogeneous polynomials using the $\maxideal$-adic topology (Equation \eqref{eq:homogeneous-polynomial-decomposition}). As the filtration topology on $\powseries{\K}{x}$ induced by a weighting $w = (w_1, \ldots, w_n)$ coincides with the $\maxideal$-adic topology, we can also write $f = \sum_{k = 0}^\infty f_k$ where $f_k \in \mathcal{P}^k$ for each $k$. Similarly, every formal vector field $X \in \vf$ can be written as an infinite sum of quasi-homogeneous vector fields:
\begin{align*}
    X = \sum_{k = -w_n}^\infty X_{[k]}, \quad \text{where } X_{[k]} \in \vfGrad{k}.
\end{align*}
(The sum starts with negative $k$ as $\coordvf{x^i}$ has weighted degree $-w_i$.)

\begin{definition}\label{def:weighted-lin}
    We call $X_{[0]}$ the \textbf{weighted linear approximation} of $X$ with respect to $w$. A formal vector field $X \in \vf$ is called \textbf{admissible} with respect to $w$ (or $w$-admissible) if $X \in \vfFilt{0}$, or equivalently, if $X_{[k]} = 0$ for all $k < 0$. We use the same terminology for a smooth vector field on an open neighbourhood of $0 \in \R^n$, referring to its formal Taylor series at $0$.
\end{definition}

\begin{remark}\label{rmk:kappa_t-pullback}
    Equivalently, thinking of the $\K$-algebra endomorphism $\kappa_t \colon \powseries{\K}{t, x} \to \powseries{\K}{t, x}$ from Remark \ref{rmk:kappa_t-laurent} as acting on Laurent series in $t$, the pullback of any $X \in \vf$ along $\kappa_t$ is given by
    \begin{align*}
        \kappa_t^*X = \sum_{k = -w_n}^\infty t^kX_{[k]}.
    \end{align*}
    This is generally not a time-dependent formal vector field in the sense of Section \ref{subsec:time-dependence} as the right-hand side can include negative powers of $t$, but $w$-admissible formal vector fields are precisely those for which $\kappa_t^*X$ \emph{does} define a time-dependent formal vector field---in fact, an \emph{evaluative} one, with $(\kappa_t^*X)|_{t = 0} = X_{[0]}$ and $(\kappa_t^*X)|_{t = 1} = X$. For a smooth vector field $X$ on $\R^n$, $w$-admissibility means that $\lim_{t \to 0^+} (\kappa_t^*X)$ exists and equals $X_{[0]}$, where $\kappa_t \colon \R^n \to \R^n$ is the diffeomorphism \eqref{eq:weighted-action}.
\end{remark}

\begin{remark}
    If $X = \sum_i X^i\coordvf{x^i}$ (formal or smooth) is admissible with respect to \emph{any} weighting, then its components satisfy $X^i(0) = 0$ for all $i$---a consequence of requiring weights to be strictly positive in Definition \ref{def:weighting}. Conversely, if $X^i(0) = 0$ for all $i$, then $X$ is admissible with respect to the trivial weighting. In Loizides--Meinrenken's differential-geometric framework of weightings \cite{loizides_differential_2023}, one can declare some coordinates to have weight $0$; these correspond to directions along a specified submanifold (for us, a single point) of a smooth manifold. One obtains quasi-homogeneous approximations of not only functions and vector fields at a point, but also tensor fields, differential forms, and multivector fields along submanifolds. We envisage generalizations of the present work to quasi-homogeneous normal forms for such objects along submanifolds.
\end{remark}

We now turn to the question of formal weighted linearizability for vector fields:
\begin{center}
    \emph{Given a $w$-admissible vector field $X \in \vfFilt{0}$ with weighted linear approximation $X_{[0]} \in \vfGrad{0}$,\\when does there exist a formal diffeomorphism $\varphi \colon \powseries{\K}{x} \to \powseries{\K}{x}$ such that $\varphi^*X = X_{[0]}$?}
\end{center}
Before giving a general answer, we revisit a particularly special case from Section \ref{subsec:background}: recall that the \emph{weighted Euler vector field} associated to $w = (w_1, \ldots, w_n)$ is given by
\begin{align}\label{eq:weighted-euler}
    \euler_w := \sum_{i = 1}^n w_ix^i\coordvf{x^i} \in \vfGrad{0}.
\end{align}
For all $\alpha \in \N^n$ and $i \in \{1, \ldots, n\}$, one has $\euler_w(x^\alpha) = \pair{w, \alpha} x^\alpha$ and $\big[\euler_w, x^\alpha\coordvf{x^i}\big] = \left(\pair{w, \alpha} - w_i\right)x^\alpha\coordvf{x^i}$. It follows that a polynomial $f \in \K[x]$ lies in $\mathcal{P}^k$ if and only if $\euler_w(f) = kf$, and a polynomial vector field $X \in \vf_\poly$ lies in $\vfGrad{k}$ if and only if $[\euler_w, X] = kX$. A $w$-admissible vector field (formal or smooth) is called \emph{weighted Euler-like} if its weighted linear approximation is equal to $\euler_w$. We obtain a formal version of \cite[Lemma 2.5]{meinrenken_euler-like_2021} over any field of characteristic zero:

\begin{proposition}\label{prop:weighted-euler-like-lin}
    If $X \in \vf$ is weighted Euler-like with respect to a weighting $w$ of $\K^n$, then there exists a formal diffeomorphism $\varphi$ of $\K^n$ such that $\varphi^*X = \euler_w$.
\end{proposition}

\begin{proof}
    By assumption, we can write $X = \sum_{k = 0}^\infty X_{[k]}$ where $X_{[k]} \in \vfGrad{k}$ and $X_{[0]} = \euler_w$. The time-dependent formal vector fields given by $X_t := \sum_{k = 0}^\infty t^kX_{[k]}$ and $U_t := \sum_{k = 0}^\infty t^kX_{[k + 1]}$ thus satisfy
    \begin{align*}
        [\euler_w, X_t] & = \sum_{k = 0}^\infty t^k[\euler_w, X_{[k]}] = \sum_{k = 0}^\infty kt^kX_{[k]} = t\f{dX_t}{dt}, \\
        X_t - \euler_w & = \bigg(\euler_w + \sum_{k = 1}^\infty t^kX_{[k]}\bigg) - \euler_w = tU_t.
    \end{align*}
    It follows that $t\left(\f{dX_t}{dt} + [U_t, X_t]\right) = 0$, hence $\f{dX_t}{dt} + [U_t, X_t] = 0$ as $\tdvf$ is torsion-free. Now observe that $X_t$ is evaluative, with $X_t|_{t = 0} = \euler_w$ and $X_t|_{t = 1} = X$; on the other hand, the flow $\varphi_t$ of $U_t$ is evaluative by Proposition \ref{prop:evaluative-flow-weighted-condition}. Thus the formal diffeomorphism $\varphi := \varphi_1$ satisfies $\varphi^*X = \euler_w$ by the $(r, s) = (1, 0)$ case of the formal Moser trick (Corollary \ref{cor:formal-moser}).
\end{proof}

To motivate the general answer, suppose that $X$ is a $w$-admissible smooth vector field on $\R^n$ with weighted linear approximation $X_{[0]}$ and $\varphi \colon \R^n \to \R^n$ is a diffeomorphism satisfying $\varphi^*X = X_{[0]}$ as smooth vector fields. Seeking a Moser-type construction for $\varphi$, we want to think of $X_{[0]}$ and $X$ as time slices of a one-parameter family of vector fields $X_t$ at $t = 0$ and $t = 1$, respectively, and $\varphi$ as the $t = 1$ time slice of a one-parameter family of diffeomorphisms $\varphi_t$. Remark \ref{rmk:kappa_t-pullback} suggests the candidate $X_t := \kappa_t^*X$ where $\kappa_t \colon \R^n \to \R^n$ is the diffeomorphism \eqref{eq:weighted-action}: as $\kappa_1 = \id_{\R^n}$, we have $X_t|_{t = 1} = X$, and $w$-admissibility means that $X_t$ extends smoothly to $t = 0$ by $X_t|_{t = 0} = X_{[0]}$. If we define $\varphi_t := \kappa_t^{-1} \circ \varphi \circ \kappa_t$ for $t > 0$, then $\varphi_1 = \varphi$ and\footnote{Upper stars reverse order here because we work with actual diffeomorphisms in this discussion, c.f. Remark \ref{rmk:pullback-notation}.}
\begin{align*}
    \varphi_t^*X_t = (\kappa_t \circ \varphi_t)^*X = (\varphi \circ \kappa_t)^*X = \kappa_t^*(\varphi^*X) = \kappa_t^*X_{[0]} = X_{[0]}, \qquad \forall t > 0,
\end{align*}
the last equality using Remark \ref{rmk:kappa_t-laurent} and the fact that $X_{[0]} \in \vfGrad{0}$. Thus $\varphi_t^*X_t$ is constant with respect to $t$. Assuming that $\varphi_t$ extends smoothly to $t = 0$, taking a time derivative yields
\begin{align}\label{eq:weighted-lin-moser}
    \f{dX_t}{dt} + [U_t, X_t] = 0,
\end{align}
where $U_t$ is the time-dependent vector field that generates $\varphi_t$ as its flow. Now observe that $\{\kappa_t\}_{t > 0}$ defines a representation of the multiplicative group $\R^+$ on $\R^n$ as $\kappa_{st} = \kappa_s \circ \kappa_t$ for all $s, t > 0$; consequently, the one-parameter family of diffeomorphisms $\{\varphi_t\}_{t > 0}$ enjoys the homogeneity property
\begin{align*}
    \varphi_{st} = \kappa_t^{-1} \circ \varphi_s \circ \kappa_t, \qquad \forall s, t > 0.
\end{align*}
For fixed $t$, differentiating both sides of this equation with respect to $s$ yields $tU_{st} = \kappa_t^*U_s$. Assuming $U_s$ is analytic with respect to the time parameter, so that $U_s = \sum_{k = 0}^\infty s^kU_{[k + 1]}$ for some time-independent vector fields $U_{[k + 1]}$ (the shifted index to be explained imminently), we obtain
\begin{align*}
    \sum_{k = 0}^\infty s^kt^{k + 1}U_{[k + 1]} = \sum_{k = 0}^\infty s^k\kappa_t^*U_{[k + 1]}.
\end{align*}
Comparing the coefficient of $s^k$ on both sides yields $\kappa_t^*U_{[k + 1]} = t^{k + 1}U_{[k + 1]}$; that is, $U_{[k + 1]} \in \vfGrad{k + 1}$. If we also assume that $X_t = \sum_{k = 0}^\infty t^kX_{[k]}$ is analytic, then Equation \eqref{eq:weighted-lin-moser} becomes
\begin{align*}
    \sum_{k = 0}^\infty t^k\Bigg((k + 1)X_{[k + 1]} + \sum_{\substack{i, j \geq 0 \\ i + j = k}} \left[U_{[i + 1]}, X_{[j]}\right]\Bigg) = 0.
\end{align*}
Thus each coefficient of $t^k$ vanishes, and we obtain a weighted version of the \emph{homological equation}
\begin{align}\label{eq:homological-equation}
    \left[X_{[0]}, U_{[k + 1]}\right] = (k + 1)X_{[k + 1]} - \sum_{i = 0}^{k - 1} \left[X_{[k - i]}, U_{[i + 1]}\right].
\end{align}
Both sides are quasi-homogeneous vector fields of weighted degree $k + 1$. In particular, if the adjoint operator $\ad_{X_{[0]}} := [X_{[0]}, \cdot] \colon \vfGrad{k + 1} \to \vfGrad{k + 1}$ is invertible, then Equation \eqref{eq:homological-equation} provides a recursive formula that uniquely determines $U_{[k + 1]}$ in terms of $U_{[1]}, \ldots, U_{[k]}$.

These heuristics assumed analyticity with respect to time, but the formal world does not concern itself with such matters of convergence. By turning the above argument on its head, we thus obtain:

\begin{theorem}\label{thm:weighted-linearizability}
    Suppose $X \in \vfFilt{0}$ is $w$-admissible, with weighted linear approximation $X_{[0]} \in \vfGrad{0}$. If $\ad_{X_{[0]}} \colon \vfGrad{k} \to \vfGrad{k}$ is invertible for all $k \geq 1$, then there exists a formal diffeomorphism $\varphi$ of $\K^n$ such that $\varphi^*X = X_{[0]}$.
\end{theorem}

\begin{proof}
    As $X$ is $w$-admissible, we have $X = \sum_{k = 0}^\infty X_{[k]}$ where $X_{[k]} \in \vfGrad{k}$. Thus $X_t := \sum_{k = 0}^\infty t^kX_{[k]}$ is an evaluative formal vector field satisfying $X_t|_{t = 0} = X_{[0]}$ and $X_t|_{t = 1} = X$. Define a time-dependent formal vector field $U_t := \sum_{k = 0}^\infty t^kU_{[k + 1]}$ by declaring $U_{[1]} \in \vfGrad{1}$ to be the unique solution to $[X_{[0]}, U_{[1]}] = X_{[1]}$, then recursively defining $U_{[k + 1]} \in \vfGrad{k + 1}$ by inverting Equation \eqref{eq:homological-equation}.

    By construction, $X_t$ and $U_t$ satisfy Equation \eqref{eq:weighted-lin-moser} in the formal sense. By Proposition \ref{prop:evaluative-flow-weighted-condition}, the flow $\varphi_t$ of $U_t$ is evaluative. The $(r, s) = (1, 0)$ case of the formal Moser trick (Corollary \ref{cor:formal-moser}) yields the desired formal diffeomorphism $\varphi := \varphi_1$.
\end{proof}

As a special case, we note that if $X$ is weighted Euler-like, then $\ad_{X_{[0]}} = \ad_{\euler_w}$ acts on $\vfGrad{k}$ by multiplication by $k$. The unique solution to Equation \eqref{eq:homological-equation} is then given by $U_{[k + 1]} = X_{[k + 1]}$, recovering the proof of Proposition \ref{prop:weighted-euler-like-lin}. In general, Theorem \ref{thm:weighted-linearizability} is advantageous as each $\vfGrad{k}$ is a finite-dimensional $\K$-vector space by \eqref{eq:quasi-homogeneous-vf}, so we can determine invertibility of the adjoint operator on $\vfGrad{k}$ by examining its eigenvalues; such is the topic of the next section.

\subsection{The weighted homological equation}\label{subsec:weighted-hom-eq}

For this section, we fix a weighting $w = (w_1, \ldots, w_n)$ of $\K^n$ once and for all. Keeping with the notation used in the proof of Proposition \ref{prop:evaluative-isotopy-weighted-condition}, we denote the number of \emph{distinct} weights in $w$ by $m$ and the number of coordinates assigned to the $\ell$\textsuperscript{th} smallest weight value by $d_\ell$. We define $\mu \colon \{1, \dots, n\} \to \{1, \dots, m\}$ by the property that $\mu(i)$ is the integer $\ell$ such that $w_i$ is the $\ell$\textsuperscript{th} smallest weight value, so that $d_\ell$ is the cardinality of $\mu^{-1}(\ell)$. We identify the linear polynomials in $\K[x]$ with the $\K$-linear dual space $(\K^n)^*$ and define a grading
\begin{align}\label{eq:dual-grading}
    (\K^n)^* = \bigoplus_{\ell = 1}^m V^\ell,
\end{align}
where $V^\ell$ is the subspace of $(\K^n)^*$ with basis given by $\{x^i \mid \mu(i) = \ell\}$, i.e., $V^{\mu(i)} = \mathcal{P}^{w_i} \cap (\K^n)^*$.

Given a (formal or smooth) vector field $X$ on $\K^n$, we denote by $X_\lin \in \vf_\poly$ the usual linear approximation of $X$, i.e., its weighted linear approximation with respect to the trivial weighting. As a derivation of $\K[x]$, $X_\lin$ preserves (unweighted) polynomial degree, hence leaves $(\K^n)^*$ invariant. As explained in Section \ref{subsec:background}, when $\K = \R$ or $\K = \C$, conditions guaranteeing $X$ to be linearizable in a given category are often phrased in terms of the eigenvalues of the Jacobian matrix $DX(0)$, where $X$ is identified with the $n$-tuple of its components. Equivalently, these are the eigenvalues of  $X_\lin \colon (\K^n)^* \to (\K^n)^*$, as the transpose of $DX(0)$ represents $X_\lin$ in the ordered basis $\{x^1, \ldots, x^n\}$.

Assuming $X$ is $w$-admissible with weighted linear approximation $X_{[0]} \in \vfGrad{0}$, we would like to phrase analogous conditions for weighted linearizability in terms of ``the eigenvalues of $X_{[0]}$ on $(\K^n)^*$''. A priori, it is not immediately clear how to interpret this, as $X_{[0]}$ generally does not preserve $(\K^n)^*$: $X_{[0]}$ preserves \emph{weighted} polynomial degree, so the image of each $x^i$ under $X_{[0]}$ may include products of \emph{unweighted} linear polynomials with weighted degree equal to $w_i$. One way to return to $(\K^n)^*$ is to simply drop the (unweighted) non-linear terms---in other words, to act on $(\K^n)^*$ via $(X_{[0]})_\lin$. Alternatively, one could first act on $(\K^n)^*$ via the usual linear approximation $X_\lin$, then drop any terms that are not weighted-linear---in other words, act on $(\K^n)^*$ via $(X_\lin)_{[0]}$. It turns out that the order does not matter:

\begin{lemma}\label{lem:linear-approximation-commutative}
    Suppose $X \in \vf$ has weighted linear approximation $X_{[0]} \in \vfGrad{0}$ and unweighted linear approximation $X_\lin \in \vf_\poly$.
    \begin{enumerate}[label=(\roman*), ref=\thetheorem{}(\roman*)]
        \item Weighted linear approximation commutes with unweighted linear approximation:
        \begin{align*}
            (X_\lin)_{[0]} = (X_{[0]})_\lin.
        \end{align*}
        \item\label{itm:weighted-lin-char-poly} If $X$ is $w$-admissible, then $(X_{[0]})_\lin$ and $X_\lin$ share the same characteristic polynomial as operators on $(\K^n)^*$, and $(X_{[0]})_\lin$ preserves the grading \eqref{eq:dual-grading}.
    \end{enumerate}
\end{lemma}

\begin{proof}
    \textbf{(i)} By definition, $X_\lin$ is obtained from $X$ by keeping only the monomial vector fields $x^\alpha\coordvf{x^i}$ with $|\alpha| = 1$, whereas $X_{[0]}$ is obtained from $X$ by keeping only the monomial vector fields with $\pair{w, \alpha} = w_i$. Thus $(X_\lin)_{[0]}$ and $(X_{[0]})_\lin$ are both obtained from $X$ by keeping only the terms satisfying both equalities.

    \textbf{(ii)} Assuming $X$ is $w$-admissible, if an (unweighted) linear term $x^j\coordvf{x^i}$ appears in $X$, then it must be the case that $w_j \geq w_i$. Since we require weights to be sorted from smallest to largest in Definition \ref{def:weighting}, it follows that the Jacobian matrix $DX(0)$ is block upper triangular:
    \begin{align*}
        DX(0) =
        \begin{pmatrix}
            A_1 & * & \cdots & * \\
            O & \ddots & \ddots & \vdots \\
            \vdots & \ddots & \ddots & * \\
            O & \cdots & O & A_m
        \end{pmatrix},
    \end{align*}
    where each $A_\ell$ is a $d_\ell \times d_\ell$ matrix. In the weighted linear approximation $X_{[0]}$, one deletes the terms $x^j\coordvf{x^i}$ in $X$ such that $w_j > w_i$, so the Jacobian matrix $DX_{[0]}(0)$ is obtained from $DX(0)$ by replacing the blocks above the diagonal with zeroes. Thus $DX(0)$ and $DX_{[0]}(0)$ have the same characteristic polynomial. The first claim follows as the transposes of these matrices respectively represent $X_\lin$ and $(X_{[0]})_\lin$ in the ordered basis $\{x^1, \ldots, x^n\}$ of $(\K^n)^*$; the second claim follows as we have just shown that $DX_{[0]}(0)$ is block diagonal, and (the transposes of) these diagonal blocks correspond precisely to how $(X_{[0]})_\lin$ acts on the subspaces $V^\ell$.
\end{proof}

\emph{For simplicity, we henceforth assume that this characteristic polynomial splits in $\K$.} If not, then pass to its splitting field; the modifications needed are straightforward. As $(X_{[0]})_\lin$ respects the grading \eqref{eq:dual-grading}, we also require its eigenvalues to be listed in a way that respects this grading:

\begin{definition}\label{def:compatible-ordering}
    An ordering of the eigenvalues $(\lambda_1, \ldots, \lambda_n)$ of $(X_{[0]})_\lin \colon (\K^n)^* \to (\K^n)^*$ is called \textbf{$w$-compatible} if each $\lambda_i$ is an eigenvalue of the restriction $(X_{[0]})_\lin \colon V^{\mu(i)} \to V^{\mu(i)}$.
\end{definition}

We emphasize that these eigenvalues are the same as that of the usual linear approximation by Lemma \ref{itm:weighted-lin-char-poly}. The weighting only matters insofar as determining what orderings are allowed in our generalization of the Poincar\'{e}--Sternberg non-resonance condition:

\begin{definition}\label{def:weighted-resonance}
    Let $\lambda = (\lambda_1, \ldots, \lambda_n) \in \K^n$ be a vector. A \textbf{resonance} with respect to $w$ is a pair $(i, \alpha)$ where $i \in \{1, \ldots, n\}$ and $\alpha \in \N^n$ are such that
    \begin{align*}
        \pair{\lambda, \alpha} = \lambda_i, && \pair{w, \alpha} > w_i.
    \end{align*}
    If such a resonance exists, then the positive integer $\pair{w, \alpha} - w_i$ is called its \textbf{weighted degree}.

    Given a $w$-admissible vector field $X \in \vfFilt{0}$ with weighted linear approximation $X_{[0]} \in \vfGrad{0}$ and an integer $k \geq 1$, we call $X$ \textbf{non-resonant of weighted degree $k$} if some (equivalently, every) $w$-compatible ordering of the eigenvalues of $(X_{[0]})_\lin$ has no resonances of weighted degree $k$. We call $X$ \textbf{non-resonant} with respect to $w$ if this holds for all $k \geq 1$.
\end{definition}

Taking the trivial weighting recovers the usual definitions of (non-)resonance, for in this case the grading \eqref{eq:dual-grading} is trivial, so every ordering of the eigenvalues is a compatible ordering. We now state the main result of this section: the formal half of Theorem \ref{thm:main-theorem}.

\begin{theorem}\label{thm:adjoint-invertibility}
    Suppose $X \in \vfFilt{0}$ is $w$-admissible, with weighted linear approximation $X_{[0]} \in \vfGrad{0}$. If $k \geq 1$ is such that $X$ is non-resonant of weighted degree $k$, then $\ad_{X_{[0]}} \colon \vfGrad{k} \to \vfGrad{k}$ is invertible. In particular, if $X$ is non-resonant with respect to $w$, then there exists a formal diffeomorphism $\varphi$ of $\K^n$ such that $\varphi^*X = X_{[0]}$ by Theorem \ref{thm:weighted-linearizability}.
\end{theorem}

\begin{proof}
    By Lemma \ref{itm:weighted-lin-char-poly}, each $V^\ell$ is invariant under $(X_{[0]})_\lin$. Concatenating Jordan bases for its restrictions to each $V^\ell$ produces an ordered Jordan basis $\{p^1, \ldots, p^n\}$ of $(\K^n)^*$ for $(X_{[0]})_\lin$ such that $p^i \in V^{\mu(i)} = \mathcal{P}^{w_i} \cap (\K^n)^*$ for each $i$. Consequently, the linear change of basis $x^i \mapsto p^i$ extends to a formal diffeomorphism $\psi \colon \powseries{\K}{x} \to \powseries{\K}{x}$ that leaves $\K[x]$ invariant and preserves both weighted and unweighted polynomial degree.\footnote{Grabowski--Rotkiewicz call a map preserving weighted polynomial degree a \emph{morphism of graded spaces}.} The same facts are true of $\varphi := \psi^{-1}$, so the pullback $\varphi^*$ leaves $\vf_\poly$ invariant and preserves both the usual unweighted grading on $\vf_\poly$ and the grading \eqref{eq:vf-poly-grading}. Since $\varphi^*X_{[0]} \in \vfGrad{0}$, it suffices to prove that $\ad_{\varphi^*X_{[0]}} = \varphi^* \circ \ad_{X_{[0]}} \circ (\varphi^*)^{-1}$ is invertible on $\vfGrad{k}$.

    All monomial vector fields $x^\beta\coordvf{x^j}$ appearing in $X_{[0]}$ satisfy $\pair{w, \beta} = w_j$, and writing $X_{[0]} = (X_{[0]})_\lin + \big[X_{[0]} - (X_{[0]})_\lin\big]$ splits $X_{[0]}$ into those with $|\beta| = 1$ and those with $|\beta| > 1$. As $\{p^1, \ldots, p^n\}$ puts $(X_{[0]})_\lin$ into Jordan canonical form, we have
    \begin{align*}
        \varphi^*(X_{[0]})_\lin = \lambda_1x^1\coordvf{x^1} + \sum_{j = 2}^n (\lambda_j x^j + \varepsilon_j x^{j - 1})\coordvf{x^j},
    \end{align*}
    where $(\lambda_1, \ldots, \lambda_n)$ is a $w$-compatible ordering of the eigenvalues of $(X_{[0]})_\lin$ and $\varepsilon_j \in \{0, 1\}$ corresponds to whether $p^j$ is an eigenvector or a generalized eigenvector of higher rank. Thus we can write $\varphi^*X_{[0]} = \Lambda + J + T$, where
    \begin{align*}
        \Lambda & := \sum_{j = 1}^n \lambda_j x^j\coordvf{x^j}, & J & := \sum_{j = 2}^n \varepsilon_jx^{j - 1}\coordvf{x^j}, & T & := \varphi^*\left(X_{[0]} - (X_{[0]})_\lin\right).
    \end{align*}
    A straightforward calculation shows that monomial vector fields $x^\alpha\coordvf{x^i}$ are eigenvectors of $\ad_\Lambda$ with eigenvalues $\pair{\lambda, \alpha} - \lambda_i$. We claim that $\ad_J$ and $\ad_T$ are both represented by \emph{strictly} upper triangular matrices when the basis $\left\{x^\alpha\coordvf{x^i} \mid \pair{w, \alpha} - w_i = k\right\}$ of $\vfGrad{k}$ is ordered as follows:
    \begin{center}
        \emph{First, sort the monomial vector fields $x^\alpha\coordvf{x^i}$ in decreasing order of $|\alpha|$. Among those $x^\alpha\coordvf{x^i}$ with a fixed value of $|\alpha|$, sort the ones with $\coordvf{x^i}$ before the ones with $\coordvf{x^j}$ if $i > j$. Among those $x^\alpha\coordvf{x^i}$ with fixed values of $|\alpha|$ and $i$, sort the $x^\alpha$'s in decreasing lexicographical order of the multi-indices $\alpha$.}
    \end{center}
    From this claim, it follows that $\ad_{\varphi^*X_{[0]}} = \ad_\Lambda + \ad_J + \ad_T$ is represented by an upper triangular matrix whose diagonal entries are of the form $\pair{\lambda, \alpha} - \lambda_i$ where $\pair{w, \alpha} - w_i = k$. As $X$ is non-resonant of weighted degree $k$, these diagonal entries are non-zero, so $\ad_{\varphi^*X_{[0]}}$ is invertible on $\vfGrad{k}$.

    For $\ad_J$, let $j \in \{2, \ldots, n\}$ be such that $\varepsilon_j = 1$. (This is only possible if $w_{j - 1} = w_j$; otherwise $J$ would have a term $x^{j - 1}\coordvf{x^j}$ with negative weighted degree.) We compute
    \begin{align*}
        \left[x^{j - 1}\coordvf{x^j}, x^\alpha\coordvf{x^i}\right] & = x^{j - 1}\f{\p x^\alpha}{\p x^j}\coordvf{x^i} - x^\alpha\f{\p x^{j - 1}}{\p x^i}\coordvf{x^j}.
    \end{align*}
    If the first term is non-zero, then $\f{\partial x^\alpha}{\partial x^j} \neq 0$; writing $\alpha = (\alpha_1, \ldots, \alpha_n)$, this implies that $\alpha_j > 0$. Then
    \begin{align*}
        x^{j - 1}\f{\p x^\alpha}{\p x^j} = \alpha_j x^{\alpha'}, \quad \text{where }\alpha' = (\alpha_1, \ldots, \alpha_{j - 1} + 1, \alpha_j - 1, \ldots, \alpha_n).
    \end{align*}
    As $\alpha'$ is lexicographically greater than $\alpha$, the monomial vector field $x^{\alpha'}\coordvf{x^i}$ comes before $x^\alpha\coordvf{x^i}$ in the ordered basis. If the second term is non-zero, then $i = j - 1$, in which case $x^\alpha\f{\p x^{j - 1}}{\p x^i}\coordvf{x^j} = x^\alpha\coordvf{x^{i + 1}}$, which also comes before $x^\alpha\coordvf{x^i}$ in the prescribed order.

    For $\ad_T$, recall that $X_{[0]} - (X_{[0]})_\lin$ only involves monomial vector fields $x^\beta\coordvf{x^j}$ such that $|\beta| > 1$. As $\varphi$ preserves unweighted polynomial degree, $T = \varphi^*\big(X_{[0]} - (X_{[0]})_\lin\big)$ also only involves monomial vector fields $x^\beta\coordvf{x^j}$ such that $|\beta| > 1$. Thus $[T, x^\alpha\coordvf{x^i}]$ is a linear combination of monomial vector fields $x^\gamma\coordvf{x^k}$ such that $|\gamma| > |\alpha|$, all of which come before $x^\alpha\coordvf{x^i}$ in the order described above.
\end{proof}

\begin{example}
    Assign the weighting $w = (1, 2, 2, 3, 3)$ to the coordinates $(x, y, z, u, v)$ of $\K^5$. For arbitrary coefficients $a, b, c, d, e, f, g, h \in \K$, the vector field
    \begin{align*}
        X_{[0]} & = x\coordvf{x} + (ax^2 + 4y + z)\coordvf{y} + (bx^2 -4y)\coordvf{z} \\
        & \quad + (cx^3 + dxy + exz + 4u - v)\coordvf{u} + (fx^3 + gxy + hxz + u + 2v)\coordvf{v}
    \end{align*}
    is quasi-homogeneous of weighted degree $0$, and its unweighted linear approximation is
    \begin{align*}
        (X_{[0]})_\lin = x\coordvf{x} + (4y + z)\coordvf{y} - 4y\coordvf{z} + (4u - v)\coordvf{u} + (u + 2v)\coordvf{v}.
    \end{align*}
    A $w$-compatible ordering of the eigenvalues of $(X_{[0]})_\lin$ is given by $\lambda = (1, 2, 2, 3, 3)$, and a corresponding ordered Jordan basis of $(\K^n)^*$ is given by $\{x, 2y + z, y, u - v, u\}$. Thus the formal diffeomorphism $\psi$ from the proof of Theorem \ref{thm:adjoint-invertibility} and its inverse $\varphi$ respectively act on $\powseries{\K}{x}$ by
    \begin{align*}
        \psi \colon
        \begin{cases}
            x \mapsto x \\
            y \mapsto 2y + z \\
            z \mapsto y \\
            u \mapsto u - v \\
            v \mapsto u
        \end{cases}
        &&
        \varphi \colon
        \begin{cases}
            x \mapsto x \\
            y \mapsto z \\
            z \mapsto y - 2z \\
            u \mapsto v \\
            v \mapsto -u + v
        \end{cases}
    \end{align*}
    and it is straightforward to verify that $\varphi^*$ puts $(X_{[0]})_\lin$ into Jordan canonical form:
    \begin{align*}
        \varphi^*\left((X_{[0]})_\lin\right) = x\coordvf{x} + 2y\coordvf{y} + (y + 2z)\coordvf{z} + 3u\coordvf{u} + (u + 3v)\coordvf{v}.
    \end{align*}
    It is a happy coincidence that $\lambda = w$; this implies that no $i \in \{1, \ldots, n\}$ and $\alpha \in \N^n$ satisfy both $\pair{\lambda, \alpha} = \lambda_i$ and $\pair{w, \alpha} > w_i$, so $X_{[0]}$ is non-resonant with respect to $w$ (but resonant in the unweighted sense). Thus, any $w$-admissible vector field whose weighted linear approximation is equal to $X_{[0]}$ is related to $X_{[0]}$ via a formal diffeomorphism by Theorem \ref{thm:adjoint-invertibility}.
\end{example}

\subsection{Smooth weighted linearization}\label{subsec:weighted-lin-smooth}

We conclude by returning to our original motivating question for developing the formal machinery for Moser's trick. Let $X$ be a smooth vector field on an open neighbourhood of $0 \in \R^n$. If $X$ is admissible, non-resonant, and has weighted linear approximation $X_{[0]} \in \vfGrad{0}$ with respect to a weighting $w$ of $\R^n$, then there exists a formal diffeomorphism $\varphi$ of $\R^n$ such that $\varphi^*X = X_{[0]}$ by Theorem \ref{thm:adjoint-invertibility}. To upgrade $\varphi$ to an actual diffeomorphism using Chen's theorem of equivalence \cite{chen_equivalence_1963}, one would like to know if weighted non-resonance implies hyperbolicity. This is indeed the case, but it is not immediately clear as Definition \ref{def:weighted-resonance} involves $w$-compatible orderings. Our final results compose the smooth half of Theorem \ref{thm:main-theorem}:

\begin{lemma}\label{lem:weighted-hyperbolicity}
    Let $X$ be a smooth vector field on an open neighbourhood of $0 \in \R^n$. If there exists a weighting of $\R^n$ with respect to which $X$ is admissible and non-resonant, then $0$ is a hyperbolic fixed point of $X$.
\end{lemma}

\begin{proof}
    Suppose that $0$ is \emph{not} a hyperbolic fixed point, $w$ is any weighting with respect to which $X$ is admissible, and $\lambda = (\lambda_1, \ldots, \lambda_n)$ is any $w$-compatible ordering of the (possibly complex) eigenvalues of $(X_{[0]})_\lin$. If some eigenvalue $\lambda_j$ vanishes, then $\lambda_j = 2\lambda_j$ is a resonance with respect to $w$ as $2w_j > w_j$. If not, then there is a complex conjugate pair of non-zero purely imaginary eigenvalues $\pm i\theta$. By Lemma \ref{itm:weighted-lin-char-poly}, $(X_{[0]})_\lin$ preserves the grading $(\R^n)^* = \bigoplus_{\ell = 1}^m V^\ell$, so this pair must come from the restriction of $(X_{[0]})_\lin$ to some $V^\ell$: if any such pair came from two \emph{different} $V^\ell$'s, then the characteristic polynomial of some restriction of $(X_{[0]})_\lin$ to one of the $V^\ell$'s would have an odd number of non-real roots when counted with multiplicity, which is impossible for a polynomial with real coefficients. Thus $i\theta = 2(i\theta) + 1(-i\theta)$ is a resonance with respect to $w$.
\end{proof}

\begin{corollary}\label{cor:smooth-weighted-lin}
    Let $X$ be a smooth vector field on an open neighbourhood of $0 \in \R^n$. For any weighting of $\R^n$ with respect to which $X$ is admissible, non-resonant, and has weighted linear approximation $X_{[0]} \in \vfGrad{0}$, there exists a germ of a diffeomorphism $\varphi$ fixing $0$ such that $\varphi^*X = X_{[0]}$.
\end{corollary}

\begin{proof}
    Apply Theorem \ref{thm:adjoint-invertibility}, Lemma \ref{lem:weighted-hyperbolicity}, and Chen's theorem of equivalence.
\end{proof}

In particular, the weighted Euler vector field \eqref{eq:weighted-euler} associated to any weighting $w$ of $\R^n$ is admissible and non-resonant with respect to $w$. We thus recover Algaba--Garc\'{i}a--Reyes' and Meinrenken's result that every smooth weighted Euler-like vector field on $\R^n$ is smoothly linearizable.


\bibliographystyle{plain}
\bibliography{Linearization}

\end{document}